\PassOptionsToPackage{dvipsnames}{xcolor}
\documentclass[a4paper, 11pt]{article}

\usepackage{times, a4wide}
\usepackage{amsmath, amssymb, amsthm, mathtools, thmtools}
\usepackage{tikz}
\usepackage{enumerate}
\usepackage{dsfont}
\usepackage[OT1]{fontenc}
\usepackage{bm}
\usepackage{centernot}
\usepackage{hyperref}
\hypersetup{colorlinks=true,citecolor=blue,linkcolor=blue,urlcolor=blue}

\usepackage[dvipsnames]{xcolor}

\newtheorem{theorem}{Theorem}
\newtheorem{definition}[theorem]{Definition}

\newtheorem{corollary}[theorem]{Corollary}
\newtheorem{observation}[theorem]{Observation}
\newtheorem{lemma}[theorem]{Lemma}

\newtheorem*{cordecomp}{Lemma~\ref{cor:decomp}}
\newtheorem*{lemnonrecoalt}{Lemma~\ref{lem:nonrecoalt}}
\newtheorem*{lemdegreesum}{Lemma~\ref{lem:degree-sum}}
\newtheorem*{lemmarginal}{Lemma~\ref{lem:marginal}}
\newtheorem*{lemPoissonneig}{Lemma~\ref{lem:Poissonneig}}
\newtheorem*{lemspectologSob}{Lemma~\ref{lem:spectologSob}}

\theoremstyle{definition}
\newtheorem{remarkone}[theorem]{Remark}

\newcommand{\Poisson}{\mathsf{Poi}}

\newcommand{\pl}{\textup{\texttt{+}}}
\newcommand{\mi}{\textup{\texttt{-}}}
\newcommand{\ord}{\mathrm{ord}}
\newcommand{\dis}{\mathrm{dis}}

\newcommand{\good}{\mathrm{good}}
\newcommand{\bad}{\mathrm{bad}}

\renewcommand{\P}{\mathbb{P}}
\newcommand{\emm}{\mathrm{e}}

\DeclareMathOperator{\Var}{Var}
\DeclareMathOperator{\Ent}{Ent}
\DeclareMathOperator{\E}{\mathbb{E}} 

\def\Vin{V_{\mathrm{in}}}

\newcommand{\cE}{\mathcal{E}}
\newcommand{\cT}{\mathcal{T}}
\newcommand{\cS}{\mathcal{S}}
\newcommand{\trel}{t_{\mathrm{rel}}}
\newcommand{\trelct}{t_{\mathrm{rel-ct}}}
\def\p{\mathbf{p}}

\let\epsilon = \varepsilon
\newcommand{\poly}{\mathrm{poly}}
\renewcommand{\Pr}{\P}
\def\good{\mathrm{good}}

\def\tree{\mathbb{T}}
\def\down{\mathrm{down}}

\def\Downg{\mathcal{D}_\good}
\def\Downgbig{\mathcal{D}_{\good,\geq D}}
\def\Downb{\mathcal{D}_\bad}
\def\Downgsmall{\mathcal{D}_{\good,< D}}
\def\dom{\mathrm{dom}}

\title{Fast Mixing for Low-Temperature Potts Models via Poisson Trees}

\author{Zongchen Chen\thanks{School of Computer Science, Georgia Institute of Technology, Atlanta, GA, USA.} \and Andreas Galanis\thanks{Department of Computer Science, University of Oxford, Oxford, UK.} \and Leslie Ann Goldberg\footnotemark[2] \and Xandru Mifsud\textsuperscript{\S}\footnotemark[2] \and Paulina Smolarova\footnotemark[2] \and Xusheng Zhang\textsuperscript{\S}\thanks{Department of Computer Science, Durham University, Durham, UK.}}
\date{July 22, 2026}

\begin{document}

\maketitle

\begingroup
\renewcommand{\thefootnote}{\S}
\footnotetext{This work was supported by the Clarendon Fund Scholarship awarded to Xandru Mifsud, and the EPSRC New Investigator Award UKRI155 awarded to Xusheng Zhang.}
\endgroup

\begin{abstract}
The $q$-state ferromagnetic Potts model on a graph $G$ is a probability distribution on the set of all $q$-colourings of $G$ that favours configurations with many monochromatic edges. Approximate sampling from the Potts model is a central algorithmic problem in the study of spin systems on sparse graphs. Of particular interest is the so called low-temperature regime, where the model strongly favours ordered configurations (in which one of the colours dominates), often creating bottlenecks that make Markov-chain sampling inefficient or difficult to analyse.

Key to understanding Markov chains on large sparse graphs is how they behave on the local neighbourhood of a vertex, under the influence of the surrounding boundary condition. A well-studied  example is the random regular graph, where  local neighbourhoods are captured by regular trees; the effect of boundary conditions on such trees and their connection to mixing time is by now well~understood.

Here, we focus instead on the sparse random graph $G(n,d/n)$. The local neighbourhoods of $G(n,d/n)$ are also  tree-like, but the relevant underlying graph is a Poisson Galton--Watson tree. This motivates the study of Glauber dynamics for the low-temperature Potts model on such trees with monochromatic boundary conditions. The Poisson setting introduces difficulties that are absent from the regular case: degrees fluctuate, long induced paths may appear, and branches can terminate before reaching the boundary. As a result, the effect of the monochromatic boundary at the leaves is much less uniform.

Our main  result shows near-linear mixing for the Glauber dynamics on Poisson trees with monochromatic boundary conditions. This extends the corresponding regular-tree results of Martinelli, Sinclair, and Weitz (SODA 2004) and of Blanca, Chen, Stefankovič, and Vigoda (RANDOM 2021) to the irregular trees arising from sparse random graphs. Our proof introduces an adaptive block decomposition of the tree, built around regions containing large regular subtrees, and combines it with correlation-decay estimates  and functional-inequality arguments.

As an application, we obtain a near-linear-time approximate sampling algorithm for the Potts model on $G(n,d/n)$ at all temperatures, speeding up  significantly the best previous algorithm of Galanis, Goldberg, and Smolarova (ICALP 2025). The main new ingredient is a refined analysis of the low-temperature regime, building on the Poisson tree result. On the way, we prove a simple lifting lemma, based on the Edwards--Sokal  coupling,  that transfers entropy-factorisation bounds from the Potts model to the random-cluster model on arbitrary graphs of max degree $\Delta$.
\end{abstract}

\section{Introduction}
The  Potts model is one of the central spin systems in probability, statistical physics, and theoretical computer science, and it may be viewed as a soft-constraint generalization of graph colouring.
We study sampling from the Potts model on sparse graphs. 
Given a graph $G=(V,E)$, an integer $q\ge 2$, and a parameter $\beta>0$, the Potts distribution $\mu=\mu_{G;q,\beta}$ assigns to each colouring $\sigma\in [q]^V$ (not necessarily proper) weight
\[
\mu(\sigma) = \frac{1}{Z}\exp\Bigl(\beta \sum_{\{u,v\}\in E}\mathbf{1}\{\sigma_u=\sigma_v\}\Bigr)
,
\]
where $Z$ is the normalizing constant that makes $\mu$ a probability distribution.
The parameter $\beta$ is known as the inverse temperature; note that, in the so-called low-temperature regime, that corresponds to having $\beta$ large, the distribution favours colourings $\sigma$ with many monochromatic~edges.

For $\varepsilon>0$, we say an algorithm outputs an $\varepsilon$-sample from $\mu$ if its output distribution is within $\varepsilon$ of $\mu$ in total variation distance.  Markov chains are natural algorithms for approximate sampling from Gibbs distributions, and the Potts model in particular. 
The most natural chain, and the one studied in this paper, 
is the Glauber dynamics which starts at an arbitrary $\sigma^0\in [q]^V$ and at times $t=1,2,\hdots$ obtains $\sigma^{t+1}$ from $\sigma^t$ by updating the colour of a single vertex (keeping the others the same). It does this by choosing a vertex $v\in V$ uniformly at random and sampling $\sigma^{t+1}_v$ according to the conditional distribution induced by its neighbours. We are interested in the mixing time $T_{\mathrm{mix}}(\epsilon)$ of the chain
which, roughly, captures the time $t$ needed to have an $\epsilon$-sample $\sigma_t$  from $\mu$.

For $q>2$, the sampling problem is already subject to hardness barriers at low temperatures due to certain phase transition phenomena that in the worst case lead to intractability \cite{GalanisStefankovicVigodaYang2016,GoldbergJerrum2012}. So, to obtain sampling results at low temperatures it is natural to focus on more structured graphs.  Here we focus on sparse random graphs that exhibit many of the relevant phase transitions that make the problem computationally hard, but nevertheless progress can be made via more delicate approaches.

As a running example, we will use the   $G(n,d/n)$ random graph for constant real $d>1$ but analogous results also hold for example for the random $d$-regular graph (when $d\geq 3$ is an integer). While at high-temperatures (low $\beta$) a typical colouring from the Potts distribution is close to uniform (in the sense that the colour counts are roughly $n/q$), at low temperatures (high $\beta$) the Potts distribution is dominated by $q$ modes known as ordered phases, that roughly correspond to the $q$ monochromatic colourings. This switch from unimodal to multimodal causes Glauber exponential time to mix since its very unlikely to move from one monochromatic phase to another using single-vertex updates; this phenomenon has been detailed for the random regular graph in \cite{helmuth2023finite,CojaOghlanGalanisGoldbergRavelomananaStefankovicVigoda2023}.

Nevertheless, it is conceivable that Glauber mixes well within each of the phases, and there has been a lot of progress in recent years capturing the underlying mechanics. The key idea here is to understand how Glauber behaves on the local neighbourhood of a vertex, under the influence of the surrounding monochromatic (or almost monochromatic) condition, and relating this to the true evolution. For sparse random graphs, the neighbourhoods resemble trees so the natural object to study here is the behaviour of Glauber on trees where the leaves are conditioned to have the same colour.

This local neighbourhood approach has led to fast sampling algorithms at low temperatures that apply when $q$ is relatively large to $d$, see \cite{BGPZ2026, Galanis2026PlantingModel} for the state-of-the-art results on the random $d$-regular graph and \cite{GGS} for the $G(n,d/n)$. One of the key components in these works is the use of mixing time results that are known on the underlying tree structure under the monochromatic boundary condition. In the case of the $d$-regular tree (which is the relevant tree for the random $d$-regular graph), the required pieces are very well understood by a classical work of Martinelli, Sinclair and Weitz \cite{MSW04} and a more recent follow-up by Blanca, Chen, \v{S}tefankovic, and Vigoda~\cite{blanca2023swendsen}.

By contrast, in the case of $G(n,d/n)$ which is our main concern here, the known mixing time results on the underlying tree is far from an analogous understanding and have so far been based on crude arguments that unfortunately lead to loose mixing time bounds. Indeed, the running time of the algorithm in \cite{GGS} scales as $n^{O(\log q)}$ at low-temperatures $\beta$ as a result of the suboptimal bounds. One of the key reasons behind this difficulty is that, instead of a single tree as in the regular case, the relevant trees in $G(n,d/n)$ have a non-trivial distribution, captured by a Galton--Watson tree with Poisson offspring distribution with parameter $d$; for simplicity we refer to this distribution as the Poisson tree with parameter $d$.

Our main contribution is to obtain fast mixing results on Poisson trees, which in turn leads to a significant speed-up of the running time in the algorithm in \cite{GGS}, namely from $n^{O(\log q)}$ down to $n^{1+\delta}$ for any arbitrarily small $\delta>0$. We start by first stating  the application on $G(n,d/n)$. Here, ``whp'' is a shorthand for ``with probability $1-o(1)$ as $n$ grows large''.
\begin{theorem}\label{thm:main2}
Let $\delta>0$ be arbitrarily small and $d$ be a large enough real. Then, for all sufficiently large integers $q$ and any real $\beta>0$, there is a randomised algorithm $\mathcal{A}$ such that the following holds whp over $G\sim {G}(n,d/n)$.

    On input $G$ and any $\epsilon\geq \mathrm{poly}(1/n)$, the algorithm $\mathcal{A}$ outputs an $\varepsilon$-sample from the Potts distribution $\mu_{G;q,\beta}$ in time $O( n^{1+\delta})$.
\end{theorem}

As noted above, the main engine to obtain the fast algorithm in Theorem~\ref{thm:main2} is to obtain mixing-time results on the Poisson tree with parameter $d$ under the monochromatic boundary condition.
To introduce this result we need some notation. Let $T$ be a rooted tree with root $r$,  and for an integer $h\geq 1$, let $T^h$ be the tree obtained by truncating $T$ at level $h$. 
We will  be interested in the monochromatic boundary condition on the leaves of $T^h$ that are at distance $h$ from the root.\footnote{Note that there might not be any such leaves of $T^h$ at depth $h$, in which case
there is no boundary condition. This event has $\Omega(1)$ probability for a
supercritical Poisson tree, but it is not problematic: conditional
on extinction the process is subcritical, so the resulting tree is
$O(\log n)$ with high probability. More generally, our
upcoming block decomposition handles uniformly the cases where few branches reach level
$h$; see Lemmas~\ref{lem:localspectral} and~\ref{lem:degree-sum}.} For notational convenience we denote the common boundary colour $i$ by ``$\pl$'' and write the resulting Potts distribution as $\mu^\pl \coloneqq \mu^\pl_{T^h;q,\beta}$, that is, for $\sigma\in [q]^{V}$, 
\[
\mu^\pl(\sigma)=\mu_{T^h;q,\beta}\bigl(\sigma \,\big|\, \sigma_v=\pl \text{ for all } v  \text{ with } \mathrm{dist}(r,v)=h\bigr).
\]
In the case of a Poisson tree $T$, the truncated tree  $T^h$ will typically also  have other leaves that are much closer than $h$ to the root: these are not part of the boundary condition. 
This is one central requirement for our $G(n,d/n)$ application, and is  one of the main difficulties in our setting relative to the $d$-regular case. More generally, Poisson trees have  pieces that are very sparse (e.g. induced paths of order $\Theta(\log n)$ or other small degree induced subgraphs) and dead-ends (i.e., branches that stop before reaching level~$h$), making the effect of the boundary much less uniform and harder to capture than in the regular-tree setting.  
\begin{theorem}\label{thm:main1}
    Let $\theta>1$ and $\kappa,K>0$ be arbitrary reals.  Then, for all large enough real $d$,  for  all integer $q\geq d^{\theta}$, for all real $\beta > (1+\frac{1.01}{\theta})\frac{\log q}{d}$, the following holds for all large~enough~$n$.

      Let $T$ be a Poisson tree with parameter $d$,  truncated at level~$h=\lfloor K\log_d n\rfloor$. Then, with probability $\geq 1 -n^{-\kappa}$ over  $T$,  the Glauber dynamics for the Potts distribution $\mu^{\pl}_{T^h;q,\beta}$ with monochromatic boundary condition has mixing time $\leq |V(T)|n^{o(1)}$.
\end{theorem}

\begin{remarkone}
For  $q$ large (by taking $\theta$ large), the range of $\beta$ goes all the way down to the ``uniqueness threshold'' $\beta_u(q,d) \approx \frac{\log q}{d}$; roughly, for smaller values of $\beta$, the effect of any boundary condition dies out and the mixing time is dictated by the ``free boundary condition'' where no leaves are fixed. See~\cite{BlancaGheissari2023} for more detail and results in the regime~$\beta<\beta_u(q,d)$.

In our $G(n,d/n)$ application, we will be primarily interested in the range $\beta \geq  (2-o(1))\frac{\log q}{d}$ that captures the ordered regime where the modes corresponding to monochromatic colourings start to dominate the Potts distribution. In this regime, Theorem~\ref{thm:main1} only requires  $q\geq d^{1+o(1)}$.
\end{remarkone}

\begin{remarkone}
\label{remark:running time}
As a main step in our proof, we obtain that the spectral gap and log-Sobolev constants are both  $\geq 1/(|V(T^h)|n^{o(1)})$, see Section~\ref{sec:proofoverview} for further details. The $n^{o(1)}$ factor in all these results scales in fact as $n^{O(\frac{\beta}{\log \log n})}$.
Conditioned on survival to depth $h$, we have with high probability $|V(T)|=n^{\Omega(1)}$, so the mixing time bound is $|V(T)|^{1+o(1)}$.
      This improves significantly over the
\(|V(T)|^{O(\beta)}\) upper bound (especially when $\beta$ is large) from the classical recursive comparison argument for trees (see, e.g.,~\cite{MosselSly09}), which controls the spectral gap in terms of the maximal degree sum along a root-to-leaf path.
\end{remarkone}


\subsection{Further Related Work}\label{sec: related work}
For the ferromagnetic Ising model (the case $q=2$), Jerrum and Sinclair, and later Guo and Jerrum, gave polynomial-time sampling algorithms on arbitrary graphs at all temperatures \cite{JerrumSinclair1993,GuoJerrum2018}. By contrast, for $q>2$ approximate sampling on arbitrary graphs is already subject to \#BIS-hardness barriers  at low temperatures \cite{GoldbergJerrum2012,GalanisStefankovicVigodaYang2016}, so as noted in the introduction positive results must exploit additional graph structure. 

Low-temperature sampling for the Potts model was initiated in   \cite{JKP20,helmuth2019algorithmic} using the polymer method, see \cite{helmuth2023finite,UG1, CDFKPY24,galanis2021fast,CGGPSV21} for various follow-ups based on this algorithmic approach. The Glauber dynamics analogue of these results was initiated in \cite{Gheissari2022}, see also \cite{Gheissari2022,BGPZ2026,BGnew, BG24, GGS23,Galanis2026PlantingModel,RClattice}. Most of these results apply to lattices or sparse graphs that have expansion properties (the latter includes random graphs). We emphasize here that Glauber dynamics at low temperatures is provably slow, so the main approach behind these works is to show mixing within each of the $q$  phases (or some suitable analogue in the random cluster representation, see Section~\ref{sec: application} for relevant details),  and then  combine the corresponding samples appropriately.

The $G(n,d/n)$ setting is more delicate than the complete graph and the random $d$-regular setting
because the local neighbourhoods are non-uniform and the expansion properties are a bit weaker (e.g., the graph is disconnected), so understanding what happens for low temperatures has extra challenges. The high temperature case (small $\beta$) is better understood by now, see \cite{BlancaGheissari2023,MosselSly09,bezakova2024fast, efthymiou2023mixing,liu2024fast} for some approaches in this direction.  For the low-temperature setting, the most relevant work is  \cite{GGS} where they obtained poly-time sampling for the Potts and random-cluster models on ${G}(n,d/n)$ for sufficiently large $q$ and $d$, at all temperatures. We build upon their algorithm and improve their analysis using the Poisson tree bounds. We give more details in Section~\ref{sec: application}.

\section{Proof of Main Theorems}\label{sec:proofoverview}
Here we give an overview of the main pieces needed to obtain Theorems~\ref{thm:main2} and~\ref{thm:main1}. We start with the proof of Theorem~\ref{thm:main1} which is the main new technical contribution.

Our starting point (Section~\ref{sec:nonreco43}) is a spectral-gap bound on general trees based on  a suitable block decomposition  where the blocks exhibit the non-reconstruction property.  The approach  is largely inspired by the approach of Martinelli, Sinclair, and Weitz~\cite{MSW03, MSW04} in the case of $d$-regular trees. The main difference from their setting is that on the $d$-regular tree it is enough to consider a uniform block decomposition of the tree (subtrees of depth $\ell$), so the non-reconstruction property was inherited automatically from that of the $d$-regular tree. Instead, for us the blocks are defined so that they adapt to the structure of the tree and  exhibit the non-reconstruction property. \mbox{Similar non-uniform decompositions are considered in~\cite{doi:10.1137/1.9781611977912.179}}, though the approach there relies on geometric properties of the underlying graph (balanced separators) rather than the distributional properties that we consider here~(non-reconstruction).

We then use this framework to study the mixing time on the Poisson tree in Section~\ref{sec: mixing trees}. The key idea is to consider an adaptive block decomposition of the tree, built around regions containing large regular subtrees, where showing non-reconstruction becomes quite direct. On the other hand, we need to work a bit harder to control the size of the blocks and a related sum-of-degrees property (which is used to obtain local mixing bounds and the block~radius).

The final piece (Section~\ref{sec: application}) is to apply the Poisson-tree result to the $G(n,d/n)$ setting and obtain Theorem~\ref{thm:main2}. Aside from coupling the $G(n,d/n)$ neighbourhoods with the Poisson-tree neighbourhoods, we also need to align with the approach in \cite{GGS} where they use the random cluster (RC) representation of the Potts model (which is an edge version of the Potts model). In particular, this involves  translating the Potts mixing result to an RC mixing result. Here, we provide a simple general tool to do this lifting (see Lemma~\ref{lem:lifting}), inspired by related approaches from \cite{blanca2021entropy,blanca2023swendsen}.

Finally, in  Section~\ref{sec:block-dstar}, we give an overview of the functional analysis tools we use and give the proofs of the key results. The appendices then include the remaining proofs.

\subsection{Mixing via Nonreconstruction}\label{sec:nonreco43}

For a rooted tree $T=(V,E)$ with root $r$, for a vertex $v$ we denote by $T_v$ the subtree of $T$ consisting of $v$ and its descendants. We say that $\mathcal{B}=\{B_v\}_{v\in T}$ is  a (rooted) \emph{block cover} of $T$ if  for each $v\in T$, it holds that $v\in B_v$,  $B_v\subseteq T_v$, and $B_v$ is rooted at $v$. For a block $B_v$, we denote by $\partial  B_v$ the descendants of $v$ that are not in $B_v$ but have their parent in $B_v$. The cover radius $R^*=R^*(\mathcal{B})$ is defined as $R^* \coloneqq
        \max_{u\in T}
        \big|\{v\in T\mid u\in B_v\cup \partial B_v\}\big|$.
        
Now, let $\nu$ be a  distribution supported on $\Omega_\nu\subseteq [q]^V$. We let $P_\nu$ denote the Glauber dynamics corresponding to $\nu$ (on the space $\Omega_\nu$).  In order to bound the mixing time,  we will  consider the spectral gap $\gamma \coloneqq \gamma(P_\nu)$ of its transition matrix and, later, the related log-Sobolev constant. Recall, the spectral gap is the difference between the largest and second largest eigenvalue (where the eigenvalues are sorted in absolute value). The mixing time is then bounded by $\frac{1}{\gamma}\log \frac{4}{\nu_{\min}}$ where $\nu_{\min}=\min_{\sigma\in \Omega} \nu(\sigma)$, see also Lemma~\ref{lem:mixingtimebounds} below.

For a configuration $\tau\in \Omega_\nu$ and a set $S\subseteq V$, we denote by $\nu^\tau_S$ the distribution on $[q]^V$ obtained from $\nu$ by conditioning the vertices in $V\backslash S$ to take the spins prescribed by $\tau$, i.e. $\nu^\tau_S(\cdot)=\nu(\cdot \mid \sigma_{V\backslash S}=\tau_{V\backslash S})$, provided  that $\nu( \sigma_{V\backslash S}=\tau_{V\backslash S})>0$ so that the conditional distribution is well-defined. We use the notation $\nu_S$ (i.e., without a superscripted $\tau$) to denote the marginal distribution on the vertices in $S$. i.e., $\nu_S(\cdot)=\nu(\sigma_S=\cdot)$. We also assume throughout that $\nu$ is a (nearest-neighbour) Gibbs distribution.\footnote{\label{fn:Gibbs}By this, we mean that 
for $\tau\in [q]^V$ and every vertex subset $S$ of $T$, the conditional distribution
$\nu_S^\tau$ depends on $\tau$ only through the spins on the external boundary $\Gamma(S)=\big\{u\in V\setminus S:\exists v\in S,\ (u,v)\in E\big\}$.}

We further say that the block cover $\mathcal{B}$ satisfies non-reconstruction with block decay $\rho$ under a measure $\nu$ on $[q]^V$ if, for every $v\in V$ and colours $i,j\in[q]$ with $\nu(\sigma_v=i),\nu(\sigma_v=j)>0$, 
\begin{equation}\label{eq:nonreco5454}
\Big\|
   \nu_{\partial B_v}\big(\,\cdot\,\mid \sigma_v=i\big)-
   \nu_{\partial B_v}\big(\,\cdot\,\mid \sigma_v=j\big)
 \big)\Big\|_{\mathrm{TV}}
 \le \rho.
 \end{equation}
 Intuitively, this says that the colour at the root \(v\) may influence the colours inside the block \(B_v\), but by the time one reaches the outer boundary $\partial B_v$, this influence has essentially~died~out.
 
 Finally we say that the block cover $\mathcal{B}$ satisfies local mixing with factor $C^*$ uniformly over all boundary conditions if, for each $v\in V$  and any boundary condition $\tau\in \Omega_\nu$, the spectral gap of the Glauber dynamics $P_{\nu^{\tau}_{B_v}}$  is at least $(C^*|B_v|)^{-1}$.
\begin{theorem}[Spectral gap from non-reconstruction and local mixing]
\label{thm:block-dstar}
Fix integer $q\ge 2$, a rooted tree $T=(V,E)$ and a Gibbs distribution $\nu$. Suppose $\mathcal{B}=\{B_v\}_{v\in T}$ is a rooted block cover of $T$ that has 
\begin{enumerate}[(i)]
    \item cover radius $R^*$,
    \item non-reconstruction with block decay $\rho\in (0,\tfrac{1}{20})$, and
    \item local mixing with factor $C^*$ uniformly over all boundary conditions.
\end{enumerate}
Then, if $10\rho R^* < 1$, the spectral gap of the Glauber dynamics  $P_\nu$ on $T$ is at least $(8R^* C^*|V|)^{-1}$.
\end{theorem}
The proof of Theorem~\ref{thm:block-dstar} is via approximate factorization of variance (which is an equivalent way of bounding the spectral gap) and is given in Section~\ref{sec:block-dstar}. Note that the spectral gap bound typically incurs a factor $\log \frac{1}{\nu_\mathrm{\min}}=\Omega(n)$ in the final mixing time bound which usually is suboptimal. To improve on this, there is by now a standard toolbox via bounding the log-Sobolev constant $\alpha(P)$ of $P$, or relatedly, the modified log-Sobolev constant $\hat\alpha(P)$ (via approximate factorization of entropy). 

The definitions of these quantities are a bit technical and based on functional inequalities; these are given for completeness in Section~\ref{sec:funpreliminaries} where the bulk of the mixing proofs are given. For now, we record the following well-known facts that is all we are going to need in the main proof overview below. For a distribution $\nu$ over $[q]^V$,   we say that $\nu$ is $b$-marginally bounded for some $b>0$ if for every $\tau\in \Omega_\nu$, $S\subseteq V$ and $v\in S$, it holds that $\min\big\{\nu^\tau_S(\sigma_v=i)\mid i\in [q] \mbox{ such that }  \nu^\tau_S(\sigma_v=i)>0\big\}\geq b$.
\begin{lemma}[See, e.g., {\cite[Fact~3.5]{CLV21a}}]\label{lem:mixingtimebounds}
Let $\nu$ be a Gibbs distribution over $[q]^{V}$ and $P$ denote the Glauber dynamics for $\nu$. Let $\nu_{\mathrm{min}}=\min_{\sigma: \nu(\sigma)>0} \nu(\sigma)$. Then the following hold:
\begin{enumerate}[(i)]
    \item $T_{\mathrm{mix}}(\epsilon)\leq \frac{1}{\gamma}\log \frac{1}{\epsilon\nu_{\mathrm{min}}}$ where $\gamma$ is the spectral gap of $P$.
    \item $T_{\mathrm{mix}}(\epsilon)\leq \frac{1}{\alpha}\big(\log\log \frac{1}{\nu_{\mathrm{min}}}+\log \frac{1}{2\epsilon^2}\big)$ where $\alpha$ is the log-Sobolev constant of $P$
    \item $T_{\mathrm{mix}}(\epsilon)\leq \frac{1}{\hat\alpha}\big(\log\log \frac{1}{\nu_{\mathrm{min}}}+\log \frac{1}{2\epsilon^2}\big)$ where $\hat\alpha$ is the modified log-Sobolev constant of $P$.    
\end{enumerate}
It holds that $\alpha\leq \hat\alpha\leq 2\gamma$. Moreover, with $K(r)=\frac{1-2r}{\log (\frac{1}{r}-1)}$, it holds that $\alpha\geq K(\nu_{\mathrm{min}})\gamma$; for $\nu$ that is $b$-marginally bounded, it holds that $\alpha\geq K(b) \hat\alpha$.
\end{lemma}
While the proof of Theorem~\ref{thm:block-dstar} could be adapted to obtain a bound on the modified log-Sobolev constant, for our applications we will not lose much by using the following somewhat cruder bound that relates the log-Sobolev constant and spectral gap on trees. A similar bound was shown in \cite[Theorem 5.7]{MSW03}, for completeness we give the proof in Section~\ref{sec:lemspectologSob}.
\newcommand{\statelemspectologSob}{Let $T=(V,E)$ be a rooted tree of height $h$ and $\nu$ be a Gibbs distribution supported on $\Omega\subseteq [q]^V$ that is $b$-marginally bounded. Let $P$ denote the Glauber dynamics for $\nu$ and $\alpha$ denote its log-Sobolev constant.

For a vertex $v\in V$ and $\tau\in \Omega_\nu$, let $P^{\tau}_{T_v}$ denote the Glauber dynamics for $\nu^\tau_{T_v}$ and let $\gamma^* = \min_{v,\tau} \big\{\frac{|T_v|}{|T|}\gamma(P^{\tau}_{T_v})\big\}$.  Then $\alpha\geq \frac{1}{2h}  \frac{1-2b}{\log ((1/b)-1)} \gamma^*$.}
\begin{lemma}\label{lem:spectologSob}
\statelemspectologSob
\end{lemma}

\subsection{Mixing time on Poisson Trees}
\label{sec: mixing trees}
To obtain a suitable block cover of Poisson trees, the guiding idea is that sufficiently dense parts of the Poisson tree should behave similarly to the regular tree, while the effect of sparse regions should be absorbed by the dense parts. To capture this, we partition the vertices of a tree  into ``good'' and ``bad'' vertices; roughly a vertex $v$ is good if the subtree rooted at $v$ contains a $D$-regular tree as a subgraph. Throughout this section we use $D=\lceil (1-\epsilon)d\rceil$ where $\epsilon>0$ will be taken to be a small constant.

\begin{definition}[$T^h$ and $h$-good]\label{def:good-bad-vertices}   Let $T$ be a tree rooted at $r$, and $h,D>0$ be integers. Let $T^h$ be the subtree induced by the vertices of $T$ at depth $\leq h$. Moreover,
\begin{itemize}
\item All vertices at depth exactly~$h$ are $h$-good.
\item A vertex at depth less than $h$ is $h$-good if
it has at least~$D$  children that are $h$-good.
\end{itemize}
When   $T$ is fixed, we use
$ V^h_\mathrm{good}$ to denote the set of $h$-good vertices in~$V(T)$.
\end{definition} 

Our interest in good vertices is justified by the following upper bound on the probability that they do not have the $\pl$ colour. Lemma~\ref{lem:marginal} is proved in Section~\ref{sec:marginal}.
\newcommand{\statelemmarginal}{Let $\xi>1$ and $\epsilon,\delta\in (0,1)$ be reals. Then, for all large enough $d$ and integers  $q$ satisfying $q^{1-\xi} \leq \frac{\delta \xi}{2(1-\delta)D}\log q$, for all real \(\beta > 0\) with \(\emm^\beta\geq q^{\frac{\xi}{(1-\delta)D}}\), the following holds for an arbitrary tree $T$ rooted at a vertex $r$ and integer height $h\geq 1$.

Let $\mu^\pl=\mu^{\pl}_{T^h;q,\beta}$ be the Potts distribution on the truncated tree $T^h$ with the monochromatic boundary condition. Let  $\hat p \coloneqq \max_{v\in V^h_\mathrm{good}\backslash\{r\}, j\in[q]} \mu\big(\sigma_v\neq \pl\mid \sigma_{\mathrm{parent}(v)}=j\big)$, then $\hat p\leq 2q^{1-\xi}$.}
\begin{lemma}\label{lem:marginal}
\statelemmarginal
\end{lemma}
 We next construct a rooted block cover of $T$, based on $h$-good vertices. 

\begin{definition}[Block cover of $T^h$]\label{def:level}
 Let $T$ be a tree rooted at $r$, and $h,D,\ell>0$ be integers. Let $T^h$ be the subtree induced by the vertices of $T$ at depth at most $h$. 
 
 For $v\in V(T^h)$, let $T^h_v$ be the subtree of $T^h$ rooted at $v$ and define $T^h_{v,\ell}$  to be the subtree of $T^h_v$ consisting of vertices $u\in V(T^h_v)$ such that the path from $v$ to $u$ contains  $\leq \ell$ vertices that are $h$-good, excluding the root $v$ from the count. 
\end{definition}

Intuitively, any path from $v$ to a descendant outside of $T^h_{v,\ell}$ stops after crossing $\ell$ good vertices. Using this and the bound in Lemma~\ref{lem:marginal} for good vertices,  we will be able to show the nonreconstruction property on $T^h_{v,\ell}$.

To bound the local mixing time of the blocks, we will use the following lemma. While the technical proof is relatively standard, it has a rather important role in our argument since, for a block of size $N=(\log n)^{O(1)}$, it will allow us to improve upon a ``naive'' $n^{O(1)}$ mixing time bound which can be obtained by a recursive root-to-leaf path argument. The latter bound is exponential in the depth of the tree (which in our case can be as large as $\Theta(\log n)$). Instead, the main insight behind Lemma~\ref{lem:localspectral} is to use a centroid decomposition of the underlying block that significantly reduces  the ``depth'' of the recursion in the argument and hence obtain an $N n^{o(1)}$ upper bound on the mixing time.  The proof of Lemma~\ref{lem:localspectral} below is given in Section~\ref{sec:local mixing}.
\begin{lemma}\label{lem:localspectral}
Let $T$ be a tree and $T'$ be a subtree of $T$ with $N$ vertices. Suppose that for every subset $S$ of vertices with  $|S|\leq \lceil\log_2N\rceil$ it holds that $\deg_T(S)\leq W$. 

Then, for any boundary condition on the outside of $T'$,  the spectral gap of Glauber dynamics on $T'$ is at least $1/\big(Nq^{O(\log N)}\emm^{O(\beta W)}\big)$.
\end{lemma}

The final piece is the following Lemma~\ref{lem:degree-sum}, where we capture some key properties of the blocks, the proof is given in Section~\ref{sec:tree-properties}. For a tree $T$ and a subset $S$ of $V(T)$, we use $\deg_T(S)$ to denote the sum of the degrees of the vertices in~$S$.

\newcommand{\statelemdegreesum}{Let $d>1$ be sufficiently large. For any $\kappa, K, C_0>0$, there is $C_1>0$ so that the following holds for large enough~$n$.   

Let $h = \lfloor K \log_d n \rfloor$, 
$\ell \coloneqq \lfloor C_0\log\log n\rfloor $, and $N \coloneqq (6d)^\ell(\log n)^2$. Let $T$ be a Galton--Watson tree with offspring distribution $\mathrm{Poi}(d)$.
For  $v\in V(T^h)$, let 
$\cE_v \coloneqq \cE(T^h_{v,\ell})$ be the~event~that
\begin{enumerate}[(i)]
\item $|V(T^h_{v,\ell})| \leq N$ and, 
\item for every  $S\subseteq V(T^h_{v,\ell})$ with 
$|S| \leq \log_2 (2N)$, 
$\deg_{T^h}(S) \leq C_1 \frac{\log n}{\log\log n}$.
\end{enumerate}
Then   
$\Pr\Big(\bigcap_{v\in V(T^h)}\cE_v\Big) \geq 1-n^{-\kappa}$.}
\begin{lemma}  \label{lem:degree-sum}
\statelemdegreesum
\end{lemma}

We are ready to give the proof of Theorem~\ref{thm:main1}.
\begin{proof}[Proof of Theorem~\ref{thm:main1}]
Fix $\theta>1$ and $\kappa,K>0$. Choose $\xi$ such that
$1+\frac{1}{\theta} < \xi< 1+\frac{1.01}{\theta}$ and
$\epsilon,\delta\in(0,1)$ sufficiently small so that, for all
large enough $d$, with $D=\lceil(1-\epsilon)d\rceil$, the assumption
$\beta>\left(1+\frac{1.01}{\theta}\right)\frac{\log q}{d}$ 
implies $\emm^\beta\ge q^{\frac{\xi}{(1-\delta)D}}$. We use Lemma~\ref{lem:marginal} with parameter $\xi$. Note that $q^{1-\xi}\leq q^{-1/\theta}$ so by taking $d$ large we have that the condition $q^{1-\xi} \leq \frac{\delta \xi}{2(1-\delta)D}\log q$ is satisfied. Set $p \coloneqq 2q^{1-\xi}$; since $q\ge d^\theta$ and $\xi>1+1/\theta$,  by taking $d$ sufficiently large we have that $p\le 2d^{-\theta(\xi-1)}<1$.

Let $h=\lfloor K\log_d n\rfloor$ be the truncation height in the statement, and let $\ell=\lfloor C_0\log\log n\rfloor$ where $C_0$ is a sufficiently large constant that  we will choose later. Apply Lemma~\ref{lem:degree-sum} with parameters $\kappa+1, K,C_0$ to obtain a constant $C_1$ so that the conclusion therein holds. Let $N=(6d)^\ell(\log n)^2$  and $W=C_1\frac{\log n}{\log\log n}$.

Let $\mathcal E_0$ be the event that $|V(T^h)|\le n^M$, where $M=\kappa+1+2K$. Since $\E[V(T^h)|]= O( d^h)$, by Markov's inequality,
$\mathbb P(\mathcal E_0)\ge 1-n^{-\kappa-1}$ for all large enough $n$. Let
$\mathcal E_1$ be the event that, for every $v\in V(T^h)$, the event
$\mathcal{E}_v=\mathcal E(T^h_{v,\ell})$ from Lemma~\ref{lem:degree-sum} holds. We have $\mathbb{P}(\mathcal E_1)\geq 1-n^{-\kappa-1}$, so a union bound
gives $\mathbb P(\mathcal E_0\cap\mathcal E_1)\ge 1-n^{-\kappa}$. We henceforth condition on the event~$\mathcal E \coloneqq \mathcal E_0\cap\mathcal E_1$.

We apply Theorem~\ref{thm:block-dstar} as follows. Let $\Vin=\Vin(T^h)$ be the vertices of $T^h$ with depth $\leq h-1$ and let $\nu$ be distribution on $[q]^{\Vin}$ obtained by $\mu^\pl_{T^h;q,\beta}$ (i.e., after imposing the monochromatic
boundary condition on the vertices at depth $h$). The Glauber dynamics for
$\mu^\pl_{T^h;q,\beta}$ is precisely the Glauber dynamics for $\nu$ on the subtree induced by $\Vin$. For $v\in \Vin$, define $B_v \coloneqq V(T^h_{v,\ell})\cap \Vin$
and note that $\partial B_v=
        \{u\in \Vin\setminus B_v:\mathrm{parent}(u)\in B_v\}.$
Then $\mathcal B=\{B_v\}_{v\in \Vin}$ is a rooted block cover of the tree induced by $\Vin$.

We first bound the cover radius of $\mathcal{B}$. Consider any
$u\in \Vin$. If $u\in B_v$, then $v$ is an ancestor of
$u$. Likewise, if $u\in \partial B_v$, then
$\mathrm{parent}(u)\in B_v$, and hence $v$ is again an ancestor
of $u$. It follows that $\{v\in \Vin:u\in B_v\cup\partial B_v\}
   \subseteq \{v\in V_{\mathrm{in}}:v\text{ is an ancestor of }u\}.$
Since the tree induced by $\Vin$ has height at most $h-1$,
the set on the right has size at most $h$. Therefore,~$R^*\leq h$.

Next we verify non-reconstruction. Since $\mathcal E_1$ holds, applying
Lemma~\ref{lem:degree-sum} with singleton sets shows that every vertex of $T^h$ has degree at most
$W$. Consequently $|\partial B_v|\le W|B_v|\le WN$.
For every $u\in\partial B_v$, by the definition of the block boundary, the path
from $v$ to $u$, excluding $v$, contains exactly $\ell+1$ vertices that are $h$-good. Using
Lemma~\ref{lem:marginal}  along these $h$-good vertices gives the disagreement bound
\[\big\|\nu(\sigma_u=\cdot\mid \sigma_v=i)
        -
        \nu(\sigma_u=\cdot\mid \sigma_v=j) \big\|_{\rm TV}\le 2p^{\ell+1}
        \qquad\text{for all }i,j\in[q].
\]
Indeed, a disagreement can pass from $v$ to $u$ only if none of the $\ell+1$
$h$-good vertices on the path from $v$ to $u$ takes the boundary colour $\pl$; at
each such $h$-good vertex this has conditional probability at most $p$ by Lemma~\ref{lem:marginal}, irrespectively
of the colour of its parent. It follows that the expected number of disagreements on $\partial B_{v}$ is at most $2|\partial B_v|p^{\ell+1}$ and hence
\[\big\|
        \nu_{\partial B_v}(\cdot\mid \sigma_v=i)
        -
        \nu_{\partial B_v}(\cdot\mid \sigma_v=j)
        \big\|_{\rm TV}
        \leq
        2|\partial B_v|p^{\ell+1}
        \leq
        2WN p^{\ell+1}
        \eqqcolon \rho .
\]
Since $p\le 2d^{-\theta(\xi-1)}$ with $\theta(\xi-1)>1$ and $WN\leq (6d)^\ell (\log n)^3$, we have $6dp<1$ for large enough $d$, so choosing $C_0$  large
gives for large $n$ that $\rho<\tfrac{1}{20}$ and      $10\rho R^*<1$.

We finally verify local mixing. Let $v\in V_{\rm in}$ and consider the block $B_v$.
By Lemma~\ref{lem:degree-sum}, every set $S\subseteq B_v$ with $|S|\le\lceil\log_2 N\rceil$ satisfies $\deg_{T^h}(S)\le W$. By Lemma~\ref{lem:localspectral}, for every~boundary condition outside $B_v$, the
spectral gap of the Glauber dynamics on $B_v$ is 
$\geq 1/\big(Nq^{O(\log N)}\emm^{O(\beta W)}\big)=1/n^{o(1)}$.
Thus the block cover satisfies local mixing with factor~$C^*=n^{o(1)}$.

We now have all the pieces to apply Theorem~\ref{thm:block-dstar}. To obtain the mixing time bound of Theorem~\ref{thm:main1}, it remains to pass from spectral gap to log-Sobolev. We use Lemma~\ref{lem:spectologSob}: observe that the same argument as above applies actually to every descendant subtree $T^h_w$
with an arbitrary boundary condition outside it since the  event
$\mathcal E$ is inherited by descendant subtrees.  Hence, in the notation of Lemma~\ref{lem:spectologSob},
$\gamma^* \geq \frac{1}{8R^*C^*|V_{\rm in}|}\geq \frac{1}{|V_{\rm in}|\,n^{o(1)}}$.
Since every vertex of $T^h$
has degree at most $W$, every one-site conditional marginal is bounded below by $b \coloneqq q^{-1}e^{-\beta W}$
so $\log(1/b)\le \log q+\beta W=n^{o(1)}$ and $\frac{1-2b}{\log((1/b)-1)}=\frac{1}{n^{o(1)}}$. By Lemma~\ref{lem:spectologSob}, and using that the height is at most $h=O(\log n)$, we obtain that
the log-Sobolev constant $\alpha$ of $P$ is at least $\frac{1}{2h}\frac{1-2b}{\log((1/b)-1)}\gamma^*\geq        \frac{1}{|V_{\rm in}|\,n^{o(1)}}$.

Finally, we bound the mixing time using Lemma~\ref{lem:mixingtimebounds}. Since $|V(T^h)|\le n^M$ and the maximum degree is at most $W$, we obtain that $\log\log\frac{1}{\nu_{\min}}=n^{o(1)}$, so
\[ T_{\rm mix}\leq 
\frac{1}{\alpha} \Big(\log\log\frac{1}{\nu_{\min}}+O(1)
        \Big)\leq
        |V_{\rm in}|\,n^{o(1)}
        \le
        |V(T^h)|\,n^{o(1)}.
\]
This yields the desired bound, completing the proof.
\end{proof}

\subsection{Application: Fast Sampling on \texorpdfstring{$G(n,d/n)$}{G(n, d/n)} at all temperatures}
\label{sec: application}

To obtain a fast approximate sampling algorithm for the Potts model on $G(n,d/n)$, we use the random cluster (RC) representation and the corresponding RC Glauber dynamics. Using the near-optimal bound on the mixing time of the Potts Glauber dynamics on the wired Poisson tree in Theorem~\ref{thm:main1}, we improve upon the previous polynomial mixing time of the RC Glauber dynamics by Galanis, Goldberg and Smolarova~\cite{GGS}, and obtain near-linear mixing~time.

For a graph $G = (V,E)$ and real parameters $q,\beta > 0$, the random cluster (RC) model induces a Gibbs distribution $\pi_{G;q,\beta}$ on the edge subsets of $G$. For $F\subseteq E$, $\pi_{G;q,\beta}(F)$ is proportional to $q^{k(F)}(\emm^\beta-1)^{|F|}$, where $k(F)$ is the number of components in $(V, F)$. It is well-known that for integer $q$, the random cluster model is an alternative representation (also known as the FK-representation) of the Potts model, via the Edwards--Sokal coupling. In particular, we can transform a RC sample $F\sim \pi_{G;q,\beta}$ to a Potts sample by independently assigning to each connected component on $(V,F)$ a uniformly random spin.

The analogue of the Potts monochromatic boundary condition on $T^h$ is conditioning all the leaves to be in the same component (one way to think of this is to introduce an extra vertex with edges to all the leaves and condition all of the additional edges to belong to $F$). We denote the all-wired boundary distribution by $\pi^\pl_{T^h;q,\beta}$ and consider the analogue of Glauber dynamics for the random cluster model (note that this updates a randomly chosen edge conditioned on the status of all the other edges, see Appendix~\ref{sec:defprelims} for details).

We first convert the Potts mixing time bound for the monochromatic boundary condition to an RC mixing time bound for the all-wired boundary condition. We do this via a more general entropy mixing comparison on general graphs using the Edwards--Sokal coupling, see Section~\ref{sec:conversion} and Lemma~\ref{lem:lifting} therein. Analogous entropy mixing bounds from Potts to RC have been shown in \cite{blanca2021entropy} for bipartite graphs and on regular trees \cite{blanca2023swendsen} under monochromatic conditions; Ullrich \cite{Ullrich1,Ullrich2} also obtained analogous transfer results between Potts and RC in terms of spectral gap.
\begin{lemma}\label{lem:conversion}
    Let $\theta>1$ and $\kappa,K>0$ be arbitrary reals.  Then, for all large enough real $d$,  for  all integer $q\geq d^{\theta}$, for all real $\beta > (1+\frac{1.01}{\theta})\frac{\log q}{d}$, the following holds for all large~enough~$n$.

      Let $T$ be a Poisson tree with parameter $d$,  truncated at level~$h=\lfloor K\log_d n\rfloor$. Then, with probability $\geq 1 - \frac{1}{n^\kappa}$ over  $T$, 
      the mixing time for
      the Glauber dynamics for the random cluster distribution $\pi^{\pl}_{T^h;q,\beta}$ with the all-wired boundary condition
      is $\leq |E(T)|n^{o(1)}$. 
      
      In particular, both the spectral gap and log-Sobolev constants are $\geq 1/(|E(T)|n^{o(1)})$.
\end{lemma}

The following lemma will allow us to utilise Lemma~\ref{lem:conversion} in the context of the sparse random graph $G(n,d/n)$. It captures the well-known property that the neighbourhoods of vertices are tree-like (and contain a small number of cycles). The proof is given in Section~\ref{sec:gndnstruc}.

\newcommand{\statelemPoissonneig}{Let $d>1$ be a real and $n$ be sufficiently large. Let $h=\lfloor\frac{1}{5}\log_d n\rfloor$ and let  $T^h$ be an independent Poisson tree rooted at $v$ with parameter $\hat d=d+\frac{d^2}{n}$ truncated at depth $h$. 

Consider $G\sim G(n,d/n)$ and a fixed vertex $v$ in $G$. Let $T_{v,h}$ denote the BFS discovery tree started at $v$ after $h$ levels. Then, there is a coupling of $T_{v,h}$ and $T^h$ so that, with probability $\geq 1-n^{-7/6}$, $T_{v,h}$ is a subtree of $T^h$ and deleting $\leq 10$ edges from $T^{h}$ leaves $T_{v,h}$ as a connected component.}
\begin{lemma}\label{lem:Poissonneig}
\statelemPoissonneig
\end{lemma}
Note that Lemma~\ref{lem:Poissonneig} couples the BFS discovery tree $T_{v,h}$ from a vertex $v$ as a component a Poisson height $h$-tree after $O(1)$ edge deletions. The actual $h$-step neighbourhood of $v$, denoted by $B_{v,h}$, might contain a few cycles so it might be different from $T_{v,h}$, but again only by $O(1)$ edges, see Lemma~\ref{lem:treelike}.  With these pieces, we obtain the following corollary, via standard comparison bounds between \mbox{log-Sobolev constants. We give the proof in Section~\ref{sec:Paulina}.}

\begin{lemma}\label{lem:log-sob-RC-dynamics}
Let $\theta > 1$. Then, for all large enough $d$,  for all integer $q\geq (d+1)^{\theta}$ and real $\beta > (1+\tfrac{1.01}{\theta})\frac{\log q}{d}$, the following holds with probability $1-1/n^{\Omega(1)}$ over the choice of $G\sim G(n,d/n)$. 

Let $h=\frac{1}{5}\log_d n$, $v$ be any vertex of $G$, and $B_{v,h}$ be the $h$-step neighbourhood of $v$ in $G$.  Then, the log-Sobolev constant of the RC Glauber dynamics on $B \coloneqq B_{v,h}$ with the all-wired boundary condition is $\geq 1/(|E(B)| n^{o(1)})$.
\end{lemma}

\begin{proof}[Proof of Theorem~\ref{thm:main2}]
We first overview the algorithm of \cite{GGS} and then explain how we obtain an improved running time using Lemma~\ref{lem:log-sob-RC-dynamics}.
A first observation is that  it is enough to consider the giant component $C=(V_C,E_C)$ of $G$: whp, all remaining components are size $O(\log n)$ and are either trees or contain at most one cycle~\cite[Section 5]{random-graphs-janson-book}. The Potts/RC distributions tensorise over connected components, and for all non-giant components it is relatively standard to obtain an exact sample in linear time (e.g., using recursions).

The situation on sampling from the giant component $C$ of $G$ is more delicate. Let $\beta_0 < \beta_1$ be defined from $
    \emm^{\beta_0} - 1 = q^{(2-1/10)/d}$ and $\emm^{\beta_1} - 1= q^{(2 + 1/10)/d}$.
It was shown in \cite{GGS} that for all $d$ large enough and $q \geq d^{\Omega(d)}$, with high probability over $G$, the RC distribution on $C$ is a mixture of two distributions (up to an $\emm^{-\Omega(n)}$ error in TV distance), the disordered  $\pi^{\mathrm{dis}}=\pi_C(F=\cdot \mid\eta|E_C|\geq |F|)$ and the ordered $\pi^{\mathrm{ord}}=\pi_C(F=\cdot \mid (1-\eta) |E|\leq  |F|)$, where $\eta=1/1000$. Moreover, a typical configuration is disordered for all $\beta\leq\beta_0$, and ordered for all $\beta\geq \beta_1$.
For $\beta \in (\beta_0,\beta_1)$ and any $\delta > 0$, the mixture of $\pi^\dis$ and $\pi^\ord$ can be approximated using the algorithm from \cite[Theorem 2]{GGS} within any polynomial $n^{-\delta}$ factor (provided that $d$ and $q$ are large enough in terms of $\delta$) in $O(n^{1+\delta})$ time. 

The samples from $\pi^\ord$ and $\pi^\dis$ can be obtained in polynomial time using RC dynamics with suitable initialisations. In particular, the following was shown in~\cite{GGS}:
\begin{enumerate}[(i)]
\item the RC dynamics initialised at $F=\emptyset$ converges to $\pi^{\mathrm{dis}}$ \mbox{in $O(n^{1+o(1)}\log \frac{1}{\epsilon})$ steps for~$\beta<\beta_1$,}
\item the RC dynamics initialised at $F=E$ converges to $\pi^{\mathrm{ord}}$ in $\mathrm{poly}(n)\log \frac{1}{\epsilon}$ steps for~$\beta>\beta_0$,
\end{enumerate}
As an immediate corollary, we can obtain an $\epsilon$-sample from $\pi_{C}$ for all $\beta \leq \beta_0$ in $n^{1+o(1)}$ steps for all $\epsilon \geq 1/\poly(n)$.
Furthermore, for all $\beta \in (\beta_0,\beta_1)$ we can obtain an $\epsilon$-sample from $\pi^\dis$ in the same running time.

Hence, the main bottleneck for $\beta > \beta_0$ is the mixing time result for the ordered phase. The argument to obtain a polynomial mixing time in~\cite{GGS} uses two main ingredients: a property of the distribution $\pi^{\mathrm{ord}}$ called weak spatial mixing within the phase that occurs in the neighbourhood $B_{v,h}$ for each $v\in V$, and the mixing time under the all-wired boundary condition in $B_{v,h}$.
In particular (see~\cite[Theorem 26]{GGS}), subject to the weak spatial mixing within the ordered phase (which holds for all $d$ large enough and $q \geq d^{\Omega(d)}$ and $\beta >\beta_0$) the time the RC dynamics initialised from $F = E$ needs to get within $\epsilon\geq\emm^{-\Omega(n)}$ TV-distance from $\pi^\ord$
 is at most $T = 30|E| \times \max_{v\in V}  \frac{T_v}{|E(B_{v,h})|}\times O(\log \tfrac 1\epsilon)$, where
$T_v$ is the number of steps needed for the RC dynamics on $B_{v,h}$ with the all-wired boundary condition to get within $\tfrac{1}{n^{3}}$ TV-distance from $\pi^\pl_{B_{v,h};q,\beta}$.

By Lemma~\ref{lem:log-sob-RC-dynamics} with $\theta = \tfrac{101}{90}$, we get that for all $d$ large enough, all integer $q\geq (d+1)^{101/90}$ and all real $\beta \geq \beta_0 > (2 - \tfrac{1}{10})\tfrac{\log q}{d}$, whp, for all vertices $v$, the corresponding wired RC dynamics has the log-Sobolev constant $1/|E(B_{v,h})|n^{o(1)}$. By the mixing time bounds in Lemma~\ref{lem:mixingtimebounds}, we obtain that  for all $v$ it holds that $T_v\leq |E(B_{v,h})|n^{o(1)}\log n$. In combination with the fact that whp $|E| = O(n)$ and $\log\tfrac 1\epsilon = O(\log n)$, we get that $\max_{v\in V}  \frac{T_v}{|E(B_{v,h})|} = n^{o(1)}$.

Therefore, for any $\epsilon\geq 1/n^{\kappa}$, 
$T = n^{1+o(1)}$ steps suffice to reach at most $\epsilon$ TV-distance. To conclude, fix $\delta > 0$. If $\beta \in(\beta_0,\beta_1)$, we obtain the $(1\pm n^{-1/\delta})$-approximations of the proportions of the phases, $\hat p_\ord$ and $\hat p_\dis$ in $O(n^{1+\delta})$ time. Otherwise we use $\hat p_\ord = 0, \hat p_\dis = 1$ for $\beta \leq \beta_0$ and $\hat p_\ord = 1, \hat p_\dis = 0$ for $\beta\geq\beta_1$. Then with probability $\hat p_\ord$ we return a $n^{-1/\delta}$-sample from $\pi^\ord$, and with probability $\hat p_\dis$, we return a sample from $\pi^\dis$. In both cases, we use $O(n^{1+o(1)})$ steps of the RC dynamics. Each step can be simulated in time $O(n^{o(1)})$, thus the resulting running time is $O(n^{1+o(1)} + n^{1+\delta}) = O(n^{1+\delta})$ for all $n$ large enough. Moreover, the resulting sample has TV distance   $\leq n^{-\Omega(1/\delta)} + \emm^{-\Omega(n)} +  n^{-1/\delta}$ from $\pi_C$. Finally we convert the RC sample to the Potts sample via the Edwards--Sokal coupling.

We remark that while Theorem~\ref{thm:main1} and Lemma~\ref{lem:log-sob-RC-dynamics} apply for $q \geq \mathrm{poly}(d)$, Theorem~\ref{thm:main2} requires $q = d^{\Omega(d)}$. This lower bound on $q$ is inherited from the polymer techniques  in~\cite{GGS} used to control the phases and to establish the weak spatial mixing condition.
\end{proof}

\section{Reconstruction, Spectral Gap and Log-Sobolev}\label{sec:block-dstar}\label{sec:spectralgap}\label{sec:funpreliminaries}
In this section, we give the proofs of Theorem~\ref{thm:block-dstar} and Lemma~\ref{lem:conversion}, which are the key ingredients for our mixing-time bounds. We first recall some functional analysis preliminaries. 

Let $V$ be a set and $\nu$ a distribution on $[q]^V$;  let $\Omega=\Omega_\nu$ be the support of $\nu$. We write $\sigma \sim \nu$ to denote that~$\sigma$ is drawn from the distribution~$\nu$.
Given any function $f \colon \Omega \to \mathbb{R}$, 
we use   $\E_{V}(f)$ to denote the expectation of~$f(\sigma)$
when $\sigma \sim \nu$. Similarly, we use
$\Var_{V}(f)$ and $\Ent_{V}(f)$ to denote the variance of~$f(\sigma)$  
when $\sigma \sim \nu$. In particular, we have that $\Ent_{V}(f)=\E_V[f\log f]-\E_V[f]\log \E_V[f]$, which is well-defined only when $f\geq 0$; below we assume this tacitly when introducing notation involving entropy. We note that later, in Section~\ref{sec:conversion}, we have multiple distributions in play, so we will use the more explicit $\Var_\nu(f)$ and $\Ent_{\nu}(f)$ in order to make clear the underlying distribution explicit. 

Recall that for a configuration $\tau\in \Omega$ and a set $A\subseteq V$, we denote by $\nu^\tau_A$ the distribution on $\Omega$ obtained from $\nu$ by conditioning the vertices in $V\backslash A$ to take the spins prescribed by $\tau$, i.e. $\nu^\tau_A(\cdot)=\nu(\cdot \mid \sigma_{V\backslash A}=\tau_{V\backslash A})$. Given a function   $f \colon \Omega\to \mathbb{R}$,
a vertex set $A\subseteq  V$, and a 
configuration $\tau \in \Omega$, 
we use $\E_A^{\tau} f$, $\Var_A^{\tau} f$, and $\Ent_A^{\tau} f$  to denote the expectation, variance, and entropy of~$f(\sigma)$
when $\sigma \sim \nu_A^{\tau}$.
We use $\E_A^{\circ} f$ to denote the function that maps~$\tau$ to $\E_A^{\tau} f$ and, similarly, 
$\Var_A^{\circ} f$ to denote the function that maps~$\tau$
to $\Var_A^{\tau} f$ (and analogously for $\Ent_A^{\circ} f$).
So, for example, if $A\subseteq B \subseteq C\subseteq  V$,
the expression 
$\E^{\tau}_C \Var_B^{\circ} \E_A^{\circ} f$
denotes
$\E_{\sigma_1\sim \nu_C^{\tau}} 
\Var_{\sigma_2 \sim \nu_B^{\sigma_1}} 
\E_{\sigma_3 \sim \nu_A^{\sigma_2}} f(\sigma_3)$. From the definition of $\E_A^{\tau}$ and the tower property of conditional expectation, we have 
for any $B\subseteq A$ that 
$\E^\tau_A \E_B^{\circ}  f = \E^\tau_A f$. We will also use the following property of variance from~\cite[Section 7.1]{MSW03}.
For any
$f \colon \Omega \to \mathbb{R}$,
$B\subseteq A\subseteq  V$, and  
$\tau \in \Omega$, the law of total variance gives that
\begin{equation} \label{eq:var_total_law}  
\Var^\tau_A f = \E^\tau_A \Var_B^{\circ} f + \Var^\tau_A \E_B^{\circ} f.
\end{equation}

We are now ready to give more formal formulations for the spectral gap, log-Sobolev and modified Log-Sobolev constants that will be relevant for this section. In particular, the spectral gap of Glauber dynamics is the largest value of $\gamma$ such that for all $f:\Omega\rightarrow \mathbb{R}$
\[\Var_V (f)\leq \frac{1}{\gamma|V|} \sum_{v \in V}\ \E_{V} \left[\Var^\circ_v (f)\right].\]
This inequality is also known as approximate factorisation of variance with constant $C=\tfrac{1}{\gamma |V|}$. 
For the log-Sobolev and modified log-Sobolev constants $\alpha$ and $\hat \alpha$, we instead have for all $f
$
\[\Ent_V (f)\leq \frac{1}{\alpha|V|} \sum_{v \in V}\ \E_{V} \big[\Var^\circ_v (\sqrt{f})\big]\quad \mbox{and}\quad\Ent_V (f)\leq \frac{1}{\hat\alpha|V|} \sum_{v \in V}\ \E_{V} \left[\Ent^\circ_v (f)\right].\]
The latter inequality is known as approximate factorisation of entropy with constant $C=\tfrac{1}{\hat \alpha |V|}$.

We note that for a product-distribution $\nu$ on $[q]^V$ factorization of entropy and variance hold with constant $C=1$. More generally, suppose $\nu$ is the product distribution $\nu=\nu_1\otimes \cdots \otimes \nu_k$ where each of the $\nu_i$'s is a distribution supported over $\Omega_i\subseteq[q]^{V_i}$, and $V_1,\hdots, V_k$ are a partition of $V$.  Then, for any function $f$ on $\Omega$,  it holds that 
\[\Var_{\nu}(f)\leq \sum_{i\in [k]}\E_\nu[\Var^{\circ}_{V_i}(f)]\mbox{ and }\Ent_{\nu}(f)\leq \sum_{i\in [k]}\E_\nu[\Ent^{\circ}_{V_i}(f)].\]

\subsection{Variance Factorisation from Block Non-Reconstruction: Theorem~\ref{thm:block-dstar}}
In this section, we prove Theorem~\ref{thm:block-dstar}. Throughout, we will work on a tree $T$ rooted at a vertex~$r$; we will identify the tree with its vertex set when there is no danger of confusion. Let~$\nu$ be a Gibbs distribution on~$[q]^T$. We will need the following two pieces for proving Theorem~\ref{thm:block-dstar}. The first one follows from a standard variance decomposition, whereas the second is more specific to our setting and utilises the block decomposition and reconstruction more crucially.
Throughout, for a graph~$G$ rooted at a vertex~$v$, we denote by $G'$ the graph obtained from $G$ by deleting $v$. We begin with the variance decomposition,  proved in Section~\ref{sec:decomp}.
\newcommand{\statecordecomp}{Let $T$ be a rooted tree and let $S=T_v$ be the subtree of $T$ rooted at $v$ for some vertex $v\in V(T)$. Then, for any configuration $\eta\in \Omega$ and $g:\Omega\rightarrow \mathbb{R}$,
    \[\Var^{\eta}_{S}(g)\leq \sum_{w\in S} \E^{\eta}_{S}\Big[\Var^\circ_{T_w}\big(\E^\circ_{T_w'}(g)\big)\Big].\]}
    \begin{lemma}\label{cor:decomp}
    \statecordecomp
    \end{lemma}

To control the terms on the right-hand side, we use the following bound that utilises the reconstruction property. The proof is given in Section~\ref{sec:nonrecoalt}. 
\newcommand{\statenonrecoalt}{    Suppose that $\rho\in (0,\tfrac{1}{20})$ and that $\mathcal{B}=\{B_v\}_{v\in T}$ is a rooted block decomposition of~$T$ that satisfies non-reconstruction with block decay~$\rho$ under the distribution~$\nu$. 

Then,  for every $v\in T$, $\eta\in \Omega$ and $f:\Omega\rightarrow \mathbb{R}$ it holds that
    \begin{equation*}
        \Var^{\eta}_{T_v}[\E^{\circ}_{T_v'}(f)]\leq     4\E^{\eta}_{T_v}\big[\Var^{\circ}_{B_v}(f)\big]+5\rho \sum_{w\in B_{v}'\cup \partial B_v } \E^{\eta}_{T_v}\big[\Var^\circ_{T_w}[\E^\circ_{T_w'}(f)]\big].
    \end{equation*}
    }
\begin{lemma}\label{lem:nonrecoalt}
\statenonrecoalt
\end{lemma}
The intuition is that $B_v$ acts as a screen between $v$ and the rest of
the subtree. Any dependence of $\E^\circ_{T'_v}f$ on the spin at $v$ must
either be caused by fluctuations of $f$ inside $B_v$, which gives the
term $\E^\eta_{T_v}[\Var^\circ_{B_v}(f)]$, or must pass through the boundary
$\partial B_v$. The non-reconstruction decay $\rho$ gives that the latter channel
has strength at most $\rho$. The residual variance is then charged
recursively to the subtrees rooted at vertices in $B'_v\cup \partial B_v$.
The terms with roots in $B'_v$ arise after averaging over the region beyond the block, since the resulting
function may still depend on the non-root vertices of $B_v$; they are controlled by the same cover-radius accounting as the $\partial B_v$ boundary terms. With these two ingredients, we now proceed to complete the proof of Theorem~\ref{thm:block-dstar} following the template introduced in \cite{MSW03}.
\begin{proof}[Proof of Theorem~\ref{thm:block-dstar}]
From Section~\ref{sec:spectralgap}, to bound the spectral gap by $(8R^*C^*|V|)^{-1}$ it suffices to establish that, for an arbitrary function $f \colon \Omega \to \mathbb{R}$, it holds that
\begin{equation}\label{eq:goal2424}
        \Var_{T}(f)
        \leq
        8 R^*
        C^*
        \sum_{u\in T}
        \E_T\left[\Var^{\circ}_{u}(f)\right].
\end{equation}
By Lemma~\ref{cor:decomp} applied to $S=T$,  
\begin{equation}\label{eq:Pfbound}
\Var_{T}(f)\leq \sum_{v\in T} \E_T
        \left[
          \Var^{\circ}_{T_v}
          \left(\E^{\circ}_{T'_v}f\right)
        \right] \eqqcolon P(f).
\end{equation}
Taking expectation over $\eta\sim\nu$ in the inequality of Lemma~\ref{lem:nonrecoalt}  and summing over $v\in T$ gives
\begin{equation}\label{eq:intermedfwer33232}
        P(f)
        \le
        4
        \sum_{v\in T}
        \E_T
    \left[\Var^{\circ}_{B_v}(f)\right] +
        5\rho
        \sum_{v\in T}
        \sum_{w\in B_{v}'\cup \partial B_{v}}
        \E_T
        \left[
          \Var^{\circ}_{T_w}
         \left(\E^{\circ}_{T'_w}f\right)
        \right].
\end{equation}
Since the cover radius of the block decomposition $\mathcal{B}$ is $ R^*$, the double sum in the second term is at most $ R^* P(f)$.
For the first term, local mixing on $B_v$ with factor $C^*$ gives that for any $\eta\in \Omega$ it holds that $\Var^{\eta}_{B_v}(f)\leq C^*\sum_{u\in B_v}
        \E^{\eta}_{B_v}
        \big[\Var^{\circ}_{u}(f)\big]$, so by taking expectation over $\eta\sim\nu$ we obtain that
$\E_T\left[\Var^{\circ}_{B_v}(f)\right]
        \le
        C^*
        \sum_{u\in B_v}
        \E_T
        \big[\Var^{\circ}_{u}(f)\big]$.
        
Summing over $v\in T$ and using the bound $R^*$ on the cover radius,
\[
        \sum_{v\in T}
        \E_T
        \left[\Var^{\circ}_{B_v}(f)\right]
        \le
        C^*\sum_{v\in T}\sum_{u\in B_v}
        \E_T
        \big[\Var^{\circ}_{u}(f)\big]
        \leq C^*R^* D(f),
\]
where $D(f)\coloneqq
        \sum_{u\in T}
        \E_T
        \left[\Var^{\circ}_{u}(f)\right].$
Combining these upper bounds, \eqref{eq:intermedfwer33232} yields $P(f)
        \le
        4 C^*R^* D(f)+5\rho R^* P(f)$.
Since by assumption $5\rho R^*<1/2$, we obtain
$P(f)
        \leq
        8
        R^*C^*D(f)$.
Combining this with \eqref{eq:Pfbound} establishes \eqref{eq:goal2424}, hence proving the theorem.
\end{proof} 

\subsection{Proof of Lemma~\ref{lem:conversion}}\label{sec:conversion}
Inspired by the approaches in  \cite{blanca2021entropy,blanca2023swendsen} on $\mathbb{Z}^d$ and the regular tree respectively, the key step in the proof of Lemma~\ref{lem:conversion} is to lift the entropy factorisation from the spin world (Potts) to an entropy factorisation in a joint spin-edge world, and then project to the edge world (RC).  Before stating the relevant lemma, we first recall the  Edwards--Sokal coupling between the Potts and random
cluster distributions that will be crucial in the argument.

Let $G=(V,E)$ be a  graph. For a spin
configuration $\sigma\in[q]^V$, let
$M(\sigma)$ be the set of monochromatic edges under $\sigma$. The Edwards--Sokal joint measure on
spin-edge pairs $(\sigma,A)\in [q]^V\times \{0,1\}^E$ is given by $\nu_{G;q,\beta}(\sigma,A)
        \propto (\emm^\beta-1)^{|A^{\mathrm{in}}|}
        \mathbf 1\{A^{\mathrm{in}}\subseteq M(\sigma)\},$
where $A^{\mathrm{in}},A^{\mathrm{out}}$ are the sets of edges that are assigned 1 and 0 under $A$, respectively.
Equivalently, conditioned on $\sigma$, the edges are independent: an edge
$e=\{u,v\}$ is forced to be out if $\sigma_u\ne\sigma_v$, while if
$\sigma_u=\sigma_v$ it is in  with probability $\frac{\emm^\beta-1}{\emm^\beta}=1-\emm^{-\beta}$. 
Conversely, conditioned on $A$, the spins are constant on each connected
component of $(V,A^{\mathrm{in}})$, and different connected components receive independent and uniformly-distributed colours from~$[q]$. It is a well-known fact that the marginal distribution on $\sigma$ is the Potts distribution $\mu_{G;q,\beta}$, whereas the marginal distribution on $A$ is the RC distribution $\pi_{G;q,\beta}$.

We then show the following lemma.  A notation point: we use the same notation $\Ent_S^\circ$ for conditional entropy on the joint space $V\sqcup E$: thus, if $S\subseteq V\sqcup E$, then $\Ent_S^\circ$ denotes entropy obtained by resampling the coordinates in $S$ while fixing all coordinates outside $S$.
\begin{lemma}\label{lem:lifting}
Fix integer $q\geq 2$ and $\beta>0$. Let $G=(V,E)$ be a graph and 
$\nu=\nu_{G;q,\beta}$ be the Edwards--Sokal distribution.  Let $\mu=\mu_{G;q,\beta}$ and $\pi=\pi_{G;q,\beta}$ be the Potts and RC distributions. 

Suppose that $\mu$ satisfies
entropy factorisation with constant $C_{\mathrm{spin}}$. 
Then, with $C_{\mathrm{joint}} \coloneqq O(q^3  \emm^{3\beta \Delta}C_{\mathrm{spin}})$, for every function $g\ge0$ on the joint spin-edge space $[q]^V\times \{0,1\}^E$,
\begin{equation}\label{eq:jointfac}
\Ent_\nu(g) \leq C_{\mathrm{joint}}
        \bigg(
        \sum_{v\in V}
        \E_\nu\left[\Ent^{\circ}_v(g)\right]
        + \sum_{e\in E}
        \E_\nu\left[\Ent^{\circ}_e(g)\right]
        \bigg).
\end{equation}
As a corollary, $\pi$ satisfies
entropy factorisation with constant $C_{\mathrm{joint}}$. 
\end{lemma}
\begin{remarkone}
The same result holds with monochromatic/wired boundary conditions. 

More
precisely, let $G=(V,E)$ be a graph and let $S\subseteq V$. Let
$\Vin=V\setminus S$, and let $\mu^\pl=\mu^\pl_{G;q,\beta,S}$ be the Potts
measure on $\Vin$ obtained by conditioning every vertex of $S$ to a
fixed colour $\pl$. For $\sigma\in [q]^{\Vin}$, we denote by $\sigma^\pl$ the configuration on $V$ that agrees with $\sigma$  on $V\backslash S$ and assigns $\pl$ to the vertices in $S$. Let $\pi^\pl=\pi^\pl_{G;q,\beta,S}$ be the random-cluster measure with all-wired boundary
condition on $S$, namely 
$\pi^\pl(A)
\propto
(\emm^\beta-1)^{|A|}q^{k_S(A)}$ for $A\subseteq E$,
where $k_S(A)$ is the number of components of $(V,A)$ that do not intersect $S$. 

Then Lemma~\ref{lem:lifting} applies with $\mu^\pl,\pi^\pl$ in place of $\mu,\pi$ with the sum over spin-edge
updates in the space $\Vin \cup E$. Indeed,
consider the Edwards--Sokal measure on
$[q]^{\Vin}\times\{0,1\}^E$ given~by
$\nu^\pl(\sigma,A)
\propto
(\emm^\beta-1)^{|A|}
\mathbf 1\{A\subseteq M(\sigma^\pl)\}$,
where $M(\sigma^\pl)$ is the set of monochromatic edges under~$\sigma^\pl$. Its spin marginal is $\mu^\pl_{G;q,\beta}$ and its edge marginal is
$\pi^\pl$, since each open component that does not intersect $S$ may
choose any of the $q$ colours, while every open component intersecting $S$ is
forced to have colour $\pl$. 
\end{remarkone}
\begin{proof}[Proof of Lemma~\ref{lem:lifting}]
For $g\geq0$, define $\hat g(\sigma) \coloneqq \E_\nu[g\mid \sigma]$. 
We have the following entropy chain rule, with respect to the vertex marginal $\mu$ of $\nu$:
\[\Ent_\nu(g) =
        \Ent_\mu(\hat g)+
\E_\nu\left[\Ent^{\circ}_E(g)\right].
\]
Indeed, let $\phi(x)=x\log x$ and recall $\Ent[f]=\E[\phi(f)-\phi(\E(f))]$. We have
$\E_\mu[\hat g]=\E_\nu[g]$, so \[\Ent_\nu(g)-\Ent_\mu(\hat g)=\E_\nu[\phi(g)]-\E_\mu[\phi(\hat g)]=\E_{\sigma\sim\mu}\big[\E_\nu[\phi(g)-\phi(\hat g)\mid \sigma]\big]=\E_\nu\left[\Ent^{\circ}_E(g)\right],\]
where the last equality follows from  $\hat g(\sigma)= \E_\nu[g\mid \sigma]$ and recalling that $\E_\nu\left[\Ent^{\circ}_E(g)\right]$ means that we condition on the configuration ``off $E$'' (which in this setting are the vertex spins).

Conditioned on $\sigma\in [q]^V$, the edges  in the Edwards--Sokal coupling are
independent. Hence the product-distribution tensorisation of entropy gives
\begin{equation}\label{eq:productstar12}
\E_\nu\left[\Ent^{\circ}_E(g)\right]\leq
        \sum_{e\in E} \E_\nu\left[\Ent^{\circ}_e\left(g\right)\right].
\end{equation}
By the entropy factorisation inequality for $\mu$ on $\hat g$, we have
$\Ent_\mu(\hat g) \leq C_{\mathrm{spin}} \sum\limits_{v\in V}\E_\mu\left[\Ent^\circ_v(\hat g)\right].
$
We show next that for every $v\in V$ it holds that
\begin{equation}\label{eq:vg1}
\E_\mu\left[\Ent^\circ_v(\hat g)\right]\leq O(q^3  \emm^{3\beta \Delta})\Big(\E_\nu\left[\Ent^{\circ}_v\left(g\right)\right]+\sum_{e\in E_v}\E_\nu\left[\Ent^{\circ}_e\left(g\right)\right]\Big),
\end{equation}
where $E_v$ denotes the edges incident to $v$.
By summing this over $v\in V$ and combining with the above yields \eqref{eq:jointfac}, after noting that each edge is counted twice when taking the sum.\enlargethispage{\baselineskip}

To prove \eqref{eq:vg1}, we condition on the vertex spins outside $v$, namely $\sigma_{V\backslash v}=\tau_{V\backslash v}$ for some $\tau\in [q]^{V}$, and on the edge spins outside $E_v$, namely  $A_{E\backslash E_v}=B_{E\backslash E_v}$ for some $B\in \{0,1\}^E$. Under this joint conditioning, the only remaining variables are the spin $\sigma_v$ and its incident edges $A_{E_v}$. On the induced star, we will show  that
\begin{equation}\label{eq:indstar}
\Ent^{\tau,B}_{v,E_v}(g)
\leq
O( q^3  \emm^{3\beta \Delta})
\Big(
\E^{\tau,B}_{v,E_v}[\Ent_v^\circ(g)]
+
\sum_{e\in E_v}\E^{\tau,B}_{v,E_v}[\Ent_e^\circ(g)]
\Big).
\end{equation}
We briefly describe the proof for this before proceeding. Let $\rho=\nu^{\tau,B}_{v,E_v}$ and consider a block chain   $P_\star$ for $\rho$ which, at each step,
chooses with probability $1/2$ to resample $\sigma_v$ conditional on $A_{E_v}$,
and with probability $1/2$ to resample $A_{E_v}$ conditional on $\sigma_v$. Then we have  that $\gamma(P_\star)\geq 1/O(q^2 \emm^{2\beta \Delta})$ via a coupling argument: for any $\sigma_v$, an update of $A_{E_v}$ leads to  $A(e)=0$ for all $e\in E_v$ in both configurations with probability $1/O(q^2 \emm^{2\beta \Delta})$, and then an update of $\sigma_v$ leads to coupled configurations. This can then be translated crudely to a bound on the modified log-Sobolev constant (using, e.g., the comparison bounds in Lemma~\ref{lem:mixingtimebounds}). This gives $\Ent_{\rho}(g)\leq O(q^3 \emm^{3\beta \Delta})(\E_{\rho}[\Ent^\circ_v(g)]+\E_{\rho}[\Ent^\circ_{E_v}(g)])$ and then  \eqref{eq:indstar} follows by applying the product-distribution tensorisation of entropy (analogously to~\eqref{eq:productstar12}).

Using~\eqref{eq:indstar}, we can now conclude the remaining piece \eqref{eq:vg1}.  Let $\psi^\tau=\nu^{\tau}_{v,E\backslash E_v}$ denote the marginal edge distribution of $\nu$ on $E\backslash E_v$ conditioned on $\tau_{V\backslash v}$. Let also $\tau_{i}$ denote the spin configuration obtained from $\tau$ by changing the spin of $v$ to $i\in [q]$. For $i\in [q]$, consider the function $h_{\tau,B}(i)=\E^{\tau,B}_{v,E_v}[g\mid \sigma_v=i]$ which takes an average over the edges in $E_v$. Note that the edges in $E\setminus E_v$ are not incident to $v$, so their distribution conditioned on
$\sigma_{V\setminus\{v\}}=\tau_{V\setminus\{v\}}$ does not depend on the value of $\sigma_v$, so $\hat g(\tau_{ i})=\E_\nu[g\mid \sigma=\tau_i]=\E_{B\sim\psi^\tau}h_{\tau,B}(i)$. By convexity of entropy applied to the average over $B$, we therefore have that $\Ent^\tau_v(\hat g)\leq  \E_{B\sim \psi^\tau}\Ent^\tau_v(h_{\tau,B})$. 

In turn, we have the entropy contraction  $\Ent^\tau_v(h_{\tau,B})\leq \Ent^{\tau,B}_{v,E_v}(g)$. To see  this,  let  $\phi=x\log x$ and recall that $\Ent[f]=\E[\phi(f)]-\phi(\E[f])$. Observe that $\E^\tau_v h_{\tau,B}=\E^{\tau,B}_{v,E_v}(g)$, so their $\phi$-values that appear in the entropy are the same; the entropy contraction then follows by using the convexity of $\phi$ to obtain
\[\E^\tau_v(\phi(h_{\tau,B}))=
\E^\tau_v\big(\phi(\E^{\tau,B}_{v,E_v}[g\mid \sigma_v])\big)\leq \E^\tau_v\big(\E^{\tau,B}_{v,E_v}[\phi(g)\mid \sigma_v]\big)=\E^{\tau,B}_{v,E_v}(\phi(g)).
\]
Combining the entropy convexity/contraction above and then applying \eqref{eq:indstar} we obtain that
\begin{align*}
\Ent^\tau_v(\hat g) &\leq \E_{B\sim \psi^\tau}\Ent^\tau_v(h_{\tau,B}) \leq \E_{B\sim \psi^\tau}\Ent^{\tau,B}_{v,E_v}(g)  \\
&\leq O( q^3  \emm^{3\beta \Delta})\E_{B\sim \psi^\tau}\Big(\E^{\tau,B}_{v,E_v}[\Ent_v^\circ(g)+\sum_{e\in E_v}\E^{\tau,B}_{v,E_v}[\Ent_e^\circ(g)]\Big)\\
&=O( q^3  \emm^{3\beta \Delta})\Big(
\E_{\nu^\tau}[\Ent_v^\circ(g)]
+
\sum_{e\in E_v}\E_{\nu^\tau}[\Ent_e^\circ(g)]
\Big),
\end{align*}
where $\nu^\tau=\nu(\cdot \mid \sigma_{V\backslash v}=\tau_{V\backslash v})$. 
Taking expectation over $\tau\sim \mu$ yields \eqref{eq:vg1} as needed.

To conclude that $\pi$ satisfies
entropy factorisation with constant $C_{\mathrm{joint}}$, take any function $f:\{0,1\}^{E}\rightarrow \mathbb{R}_{\geq 0}$ and apply the spin-edge entropy factorisation inequality to the function $g$ given by $g(\sigma,A)=f(A)$. Then, the sum over $v\in V$ is zero since $g$ does not depend on spins, and analogously $\Ent_\nu(g)=\Ent_\pi(f)$. For any $e\in E$ we also have that $\E_\nu\left[\Ent^{\circ}_e(g)\right]\leq \E_\pi\left[\Ent^{\circ}_e(f)\right]$. Indeed, since $g(\sigma,A)=f(A)$, 
$\E_\nu[\Ent^\circ_e(g)]$ is the conditional entropy of $f(A)$
when $A_e$ is resampled conditioned on both $A_{E\setminus{e}}$ and
the spin configuration $\sigma$. By monotonicity of conditional entropy,
$\E_\nu[\Ent^\circ_e(g)]
\leq
\E_\nu\left[\Ent\left(f(A)\mid A_{E\setminus{e}}\right)\right]=
\E_\pi[\Ent^\circ_e(f)]$
where the equality uses that $\pi$ is the edge marginal of $\nu$. This concludes the desired factorisation for~$\pi$.
\end{proof}
With Lemma~\ref{lem:lifting} at hand, the proof of Lemma~\ref{lem:conversion} is just a matter of combining the pieces.
\begin{proof}[Proof of Lemma~\ref{lem:conversion}]
By Theorem~\ref{thm:main1} and Remark~\ref{remark:running time}, we have a lower bound of the log-Sobolev constant of the Potts Glauber dynamics. 
By Lemma~\ref{lem:mixingtimebounds}, we get that $\mu^\pl$ satisfies entropy factorisation with constant  $n^{o(1)}$.
Applying Lemma~\ref{lem:lifting}, we obtain that $\pi^\pl$ also satisfies entropy factorisation with constant $n^{o(1)}$ when the maximum degree of $G$ is at most ${O(\log n / \log \log n)}$, which happens with the desired probability by Lemma~\ref{lem:degree-sum}.
Together with marginal boundedness of $\pi^{\pl}$, we obtain corresponding lower bounds for both the spectral gap and the log-Sobolev constant, which also yields the mixing time for the RC dynamics.
\end{proof}

\bibliographystyle{plainurl}
\bibliography{fn}

\appendix

\section{The RC dynamics}\label{sec:defprelims}
The RC dynamics is an analogue of the Potts Glauber dynamics: the states are the subsets of $E$. The transition from $X_t$ to $X_{t+1}$ is done as follows:

\begin{enumerate}
    \item Pick a uniformly random edge $e\in E$
    \item If $e$ is a cut-edge in $(V, X_t\cup\{e\})$, then $X_{t+1} = X_t \cup\{e\}$ with probability $\hat p \coloneqq \frac{\emm^\beta - 1}{q + \emm^\beta - 1}$, and $X_{t+1} = X_t\setminus \{e\}$ otherwise.
    \item If $e$ is not a cut-edge in $(V, X_{t}\cup\{e\})$, then $X_{t+1} = X_{t}\cup\{e\}$ with probability $p \coloneqq 1-\emm^{-\beta}$, and $X_{t+1} = X_{t}\setminus \{e\}$ otherwise.
\end{enumerate}
We note that each step of the chain can be run in amortised time $O((\log n)^2)$, see, e.g., \cite{Huang2023FullyTime}.

\section{Remaining Proofs for Mixing}
Recall, for a graph~$G$ rooted at a vertex~$v$, $G'$ denotes the graph obtained from $G$ by~deleting~$v$. 
\subsection{Proof of Lemma~\ref{cor:decomp}}\label{sec:decomp}
We shall use the following generalisation of variance decomposition 
property \eqref{eq:var_total_law}.
\begin{lemma}[Variance analogue of {\cite[Lemma~3.1]{Caputo2021}}]\label{lem:var_decomposition}
        Let $G=(V,E)$ be a graph and $S\subseteq V$. Let $S_0,\hdots, S_{m+1}$ be vertex subsets such that $\emptyset = S_0 \subset S_1 \subset \dots \subset S_{m + 1} = S$. Then for  any  $\eta \in \Omega$ and $g \colon \Omega \to \mathbb{R}$,
        \[\Var^\eta_S(g) = \sum_{j = 1}^{m + 1} \E^\eta_S \Var_{S_j}^{\circ} \E_{S_{j - 1}}^{\circ}(g).\]
    \end{lemma}
The following property will also be useful (see, e.g., \cite[Lemma 15]{blanca2023swendsen}). For any
$f \colon \Omega \to \mathbb{R}$, $\tau\in \Omega$ and
$A,B\subseteq T$ 
such that the tree distance between~$A$ and~$B$ is at least~$2$, it holds that
\begin{equation} \label{eq:var_convexity} 
\Var^\tau_A \E_B^{\circ} f  \leq \E^\tau_B \Var^{\circ}_A f. 
\end{equation}
\begin{cordecomp}
\statecordecomp
\end{cordecomp}
    \begin{proof}
    Fix any $\eta\in \Omega$ and function $g: \Omega \to \mathbb{R}$.  Let $S_j$ be the set of vertices in $H$ in the $j$ lowest levels of $T$
in the ordinary graph distance, with $S_0=\emptyset$ and $S_{m+1}=S$.
By Lemma~\ref{lem:var_decomposition},
\[\Var^\eta_S(g) = \sum_{j = 1}^{m + 1} \E^\eta_S\Big[ \Var_{S_j}^{\circ} \big(\E_{S_{j - 1}}^{\circ}(g)\big)\Big]
\]
Conditioned on the complement of $S_j$, the distribution on $S_j$ is a
product over the subtrees rooted at $S_j\setminus S_{j-1}$. Therefore,
using the product-measure variance inequality, for any $\tau\in \Omega$, we have that $\Var^{\tau}_{S_j}
        \big(\E^{\circ}_{S_{j-1}}g\big)\leq \sum_{v\in S_j\backslash S_{j-1}} \Var^{\tau}_{T_v}\big(\E^{\circ}_{S_{j-1}}g\big)$, so we obtain that
\[\Var^\eta_S(g)\le
        \sum_{j=1}^{m+1}
        \sum_{w\in S_j\setminus S_{j-1}}\E^\eta_S\left[\Var^{\circ}_{T_w}\big(\E^{\circ}_{S_{j-1}}g\big)
        \right].\]
For $w\in S_j\setminus S_{j-1}$, note that $S_{j-1}$ the disjoint union of $S_{j-1}\backslash T_w$ and $S_{j-1}\cap T_w=T_w'$, so we can write 
\begin{align*}
\E^\eta_S\left[\Var^{\circ}_{T_w} \! \big(\E^{\circ}_{S_{j-1}}g\big)
\right]&=\E^\eta_S\left[\Var^{\circ}_{T_w} \! \big(\E^{\circ}_{S_{j-1}\backslash T_w}\E^{\circ}_{T_w'}g\big)
        \right]\leq
        \E^\eta_S\left[\E^{\circ}_{S_{j-1\backslash T_w}}\!\big(\Var^{\circ}_{T_w}\big(\E^{\circ}_{T_w'}g\big)\big)
        \right]\\
        &=\E^\eta_S
        \left[
          \Var^{\circ}_{T_w}
\! \big(\E^{\circ}_{T'_w}g\big)
        \right],
\end{align*}
where the  inequality follows from convexity of variance, cf. \eqref{eq:var_convexity}. Therefore,
\begin{equation*}
\Var^\eta_{S}(g)\le
        \sum_{j=1}^{m+1}
        \sum_{w\in S_j\setminus S_{j-1}}
        \E^\eta_S
        \left[
          \Var^{\circ}_{T_w}
\big(\E^{\circ}_{T'_w}g\big)
        \right]= \sum_{w\in S} \E^\eta_S
        \left[
          \Var^{\circ}_{T_w}
          \left(\E^{\circ}_{T'_w}g\right)
        \right].\qedhere
\end{equation*}
    \end{proof}

\subsection{Proof of Lemma~\ref{lem:nonrecoalt}} \label{sec:nonrecoalt}
The key step in the proof of Lemma~\ref{lem:nonrecoalt} that utilises nonreconstruction is the following.
\begin{lemma}
\label{lem:f3f34f4}
Suppose that $\rho\in (0,\tfrac{1}{20})$ and that $\mathcal{B}=\{B_v\}_{v\in \Vin(T)}$ is a rooted block decomposition of~$T$ that satisfies non-reconstruction with block decay~$\rho$ under the distribution~$\nu$. 

Then, for every $v\in T$ and  $\eta\in\Omega$, for any
function $g:\Omega\to\mathbb R$,
\[
        \Var^{\eta}_{T_v}
        \left(\E^{\circ}_{T'_v}g\right)
        \le
        4\E^{\eta}_{T_v}
        \left[\Var^{\circ}_{B_v}(g)\right]
        +
        5\rho\E^{\eta}_{T_v}
        \left[\Var^{\circ}_{T'_v}(g)\right].
\]
\end{lemma}
\begin{proof}
Fix $v\in T$, $\eta\in \Omega$, and arbitrary $g:\Omega \rightarrow \mathbb{R}$. Decompose
\[
        g =  g_1+ g_2 \mbox{ where } g_1=\E^{\circ}_{B_v}g \mbox{ and } g_2=g-\E^{\circ}_{B_v}g.
\]
Using $\Var(U+V)\le 2\Var(U)+2\Var(V)$,
we get
\begin{equation}\label{eq:95686}
\Var^{\eta}_{T_v}\left(\E^{\circ}_{T'_v}g\right)\leq 2\Var^{\eta}_{T_v}\left(\E^{\circ}_{T'_v} g_1\right)+2\Var^{\eta}_{T_v}\left(\E^{\circ}_{T'_v} g_2\right)
\end{equation}
We next show  that
\begin{align}
\Var^{\eta}_{T_v}\left(\E^{\circ}_{T'_v} g_2\right)\leq \E^{\eta}_{T_v}\left(\Var^{\circ}_{B_v}g\right),\label{eq:n78757}\\
\Var^{\eta}_{T_v}\left(\E^{\circ}_{T'_v} g_1\right)\leq 2 \rho\Var^{\eta}_{T_v}\left( g_1\right)\leq  2 \rho\Var^{\eta}_{T_v}\left( g\right).\label{eq:j786888}
\end{align}
Assuming these for the moment, \eqref{eq:95686} gives
\begin{equation}\label{eq:956861}
\Var^{\eta}_{T_v}\left(\E^{\circ}_{T'_v}g\right)\leq 2\E^{\eta}_{T_v}\left(\Var^{\circ}_{B_v}g\right)+4 \rho\Var^{\eta}_{T_v}\left( g\right)
\end{equation}
The law of total variance gives that $\Var^{\eta}_{T_v}\left( g\right)=\Var^{\eta}_{T_v}\left(\E^{\circ}_{T'_v} g\right)+\E^{\eta}_{T_v}\left(\Var^{\circ}_{T_v'} g\right)$, so plugging this in \eqref{eq:956861} we obtain after rearranging that
\[\Var^{\eta}_{T_v}\left(\E^{\circ}_{T'_v}g\right)\leq \tfrac{2}{1-4\rho}\E^{\eta}_{T_v}\left(\Var^{\circ}_{B_v}g\right)+\tfrac{4 \rho}{1-4\rho}\E^{\eta}_{T_v}\left(\Var^{\circ}_{T_v'} g\right),\]
that yields the statement in the lemma using that $4\rho<1/5$.

It remains to show \eqref{eq:n78757} and \eqref{eq:j786888}. For \eqref{eq:n78757}, we have that $\E^{\eta}_{T_v}g_2=\E^{\eta}_{T_v}(g-\E^{\circ}_{B_v}g)=0$, $\E^{\eta}_{T_v}\E^{\circ}_{T'_v}g_2=0$ so $\Var^{\eta}_{T_v}\left(\E^{\circ}_{T'_v} g_2\right)=\E^{\eta}_{T_v}\big(\E^{\circ}_{T'_v} g_2\big)^2$. By convexity, we therefore have that
\[\Var^{\eta}_{T_v}\left(\E^{\circ}_{T'_v} g_2\right)\leq \E^{\eta}_{T_v}\big(g_2^2 \big)=\E^{\eta}_{T_v}\left(\E^{\circ}_{B_v}\big(g_2^2 \big)\right)=\E^{\eta}_{T_v}\big(\E^{\circ}_{B_v}(g-\E^{\circ}_{B_v}g)^2\big)=\E^{\eta}_{T_v}\left(\Var^{\circ}_{B_v}g\right),\]
as needed. The second inequality in \eqref{eq:j786888}  follows from the law of total variance \[\Var^{\eta}_{T_v}\left( g\right)=\Var^{\eta}_{T_v}\left( \E^{\circ}_{B_v}g\right)+\E^{\eta}_{T_v}\left(\Var^{\circ}_{B_v}g\right)\geq \Var^{\eta}_{T_v}\left( \E^{\circ}_{B_v}g\right)=\Var^{\eta}_{T_v}\left( g_1\right).\] For the first inequality in \eqref{eq:j786888}, we have 
\begin{equation}\label{eq:45t5g4g400}
\Var^{\eta}_{T_v}\left(\E^{\circ}_{T'_v} g_1\right)=\Var^{\eta}_{T_v}\left(\E_{T_v}^\eta[ g_1\mid \sigma_v]\right)=\frac{1}{2}\sum_{1\leq i,j\leq q}p_ip_j(G_i-G_j)^2
\end{equation}
where for $i\in [q]$ we define $p_i=\nu^\eta_{T_v}(\sigma_v=i)$ and $G_i=\E_{T_v}^\eta[ g_1\mid \sigma_v=i]$; let also $m=\sum_{i}p_iG_i$. Note that  $v$ separates $T_v$ from the rest of $T$ and $\partial B_v$ separates $T_v\backslash B_v$ from $B_v$. By nonreconstruction on the block $B_v$ (and using the Gibbs property of $\nu$, cf. Footnote~\ref{fn:Gibbs}), we have 
\begin{align*}\Big\|
   \nu_{T_v}^\eta&\big(\sigma_{T_v\backslash B_v}=\cdot\,\mid \sigma_v=i\big)-
   \nu_{T_v}^\eta\big(\sigma_{T_v\backslash B_v}=\cdot\,\mid \sigma_v=j\big)
 \big)\Big\|_{\mathrm{TV}}\\
 &\hspace{4cm}
 \leq \Big\|
   \nu_{\partial B_v}\big(\,\cdot\,\mid \sigma_v=i\big)-
   \nu_{\partial B_v}\big(\,\cdot\,\mid \sigma_v=j\big)
 \big)\Big\|_{\mathrm{TV}}\leq \rho.
 \end{align*}
 So there is a coupling $\mu$ of  $\sigma^i\sim\nu_{T_v}^\eta(\cdot\mid\sigma_v=i)$ and $\sigma^j\sim\nu_{T_v}^\eta(\cdot\mid \sigma_v=j)$ so that $\mu\big(\sigma^i_{T_v\backslash B_v}\neq \sigma^j_{T_v\backslash B_v}\big)\leq \rho$. Note that $g_1$ does not depend on the spins of $B_v$ and $\sigma^i,\sigma^j$ agree on $T\backslash T_v$ (since they are both sampled from $\nu_{T_v}^\eta$), so we can write 
 \[G_i-G_j=\E_\mu\big[\big(g_1(\sigma^i)-g_1(\sigma^j)\big)\mathbf{1}\big\{\sigma^i_{T_v\backslash B_v}\neq \sigma^j_{T_v\backslash B_v}\big\}\big],\] 
 so  by Cauchy-Schwarz
 \[(G_i-G_j)^2\leq \E_\mu\Big[\big(g_1(\sigma^i)-g_1(\sigma^j)\big)^2\Big]\E_\mu\Big[\{\mathbf{1}\big\{\sigma^i_{T_v\backslash B_v}\neq \sigma^j_{T_v\backslash B_v}\big\}\Big]\leq \rho \E_\mu\Big[\big(g_1(\sigma^i)-g_1(\sigma^j)\big)^2\Big].\]
Note that $\big(g_1(\sigma^i)-g_1(\sigma^j)\big)^2\leq  2\big(g_1(\sigma^i)-m\big)^2+2\big(g_1(\sigma^j)-m\big)^2$, so plugging these estimates back into \eqref{eq:45t5g4g400} we obtain that
\[\Var^{\eta}_{T_v}\left(\E^{\circ}_{T'_v} g_1\right)\leq \rho\sum_{i} p_i \E_\mu\Big[\big(g_1(\sigma^i)-m\big)^2\Big]+\rho\sum_{j} p_j \E_\mu\Big[\big(g_1(\sigma^j)-m\big)^2\Big]=2\rho\Var^{\eta}_{T_v}(g_1),\]
thus completing the proof of \eqref{eq:j786888}, and hence the proof of Lemma~\ref{lem:f3f34f4}.
\end{proof}

We conclude this section with the proof of Lemma~\ref{lem:nonrecoalt}.
\begin{lemnonrecoalt}
\statenonrecoalt
\end{lemnonrecoalt}
\begin{proof}
Let $R=T_v\backslash (B_v\cup \partial B_v)$ and $g=\E^\circ_R(f)$. Since $R\subseteq T_v'$, note that $\E^\circ_{T_v'}(g)=\E^\circ_{T_v'}\E^\circ_R(f)=\E^\circ_{T_v'}(f)$, so by Lemma~\ref{lem:f3f34f4} we have
\[
\Var^{\eta}_{T_v}[\E^{\circ}_{T_v'}(f)]=\Var^{\eta}_{T_v}[\E^{\circ}_{T_v'}(g)]\leq
4\E^{\eta}_{T_v}
        \left[\Var^{\circ}_{B_v}(g)\right]
        +
        5\rho\E^{\eta}_{T_v}
        \left[\Var^{\circ}_{T'_v}(g)\right]
\]
Since $B_v$ and $R$ are separated by $\partial B_v$, they have distance at least 2 in $T$, so by convexity of variance (cf. \eqref{eq:var_convexity}) we have
\[
        \E^{\eta}_{T_v}
        \left[\Var^{\circ}_{B_v}(g)\right]
        =
        \E^{\eta}_{T_v}
        \left[\Var^{\circ}_{B_v}(\E^{\circ}_{R}f)\right]
        \le \E^{\eta}_{T_v}
        \left[\E^{\circ}_{R}\big(\Var^{\circ}_{B_v}f\big)\right]=
        \E^{\eta}_{T_v}
        \left[\Var^{\circ}_{B_v}(f)\right],
\]
so it only remains to show that
\begin{equation}\label{eq:goalrf4f4f}
\E^{\eta}_{T_v}
        \left[\Var^{\circ}_{T'_v}(g)\right]\leq \sum_{w\in B_v'\cup \partial B_v } \E^{\eta}_{T_v}\big[\Var^\circ_{T_w}[\E^\circ_{T_w'}(f)]\big].
\end{equation}
Indeed, applying Lemma~\ref{cor:decomp} for $S=T'_v$ gives, for every
boundary condition $\tau\in \Omega$,
\[
\Var_{T'_v}^{\tau}(g)
\le
\sum_{w\in T'_v}
\E_{T'_v}^{\tau}
\left[
\Var^{\circ}_{T_w}
\left(\E^\circ_{T'_w}g\right)
\right].
\]
Note that $T_v'=R\cup (B_v'\cup \partial B_v)$. For a vertex $w\in R$, we have  $T_w\subseteq R$ and $g=E_R^\circ f$ is independent of the spins in $T_w$, so $\Var^{\circ}_{T_w}
\big(\E^\circ_{T'_w}g\big)=0$. Therefore, after taking expectation over
$\eta\sim\nu_{T_v}^\eta$,
\begin{equation}\label{eq:infefer}
\E^\eta_{T_v}
\left[
\Var^{\circ}_{T'_v}(g)
\right]
\leq
\sum_{w\in B_v'\cup \partial B_v}
\E^\eta_{T_v}
\left[
\Var^{\circ}_{T_w}
\left(\E^\circ_{T'_w}g\right)
\right].
\end{equation}
 Now consider an arbitrary $w\in B_v'\cup \partial B_v$ and set $Q_w \coloneqq R\setminus T_w$. Since $R\cap T_w\subseteq T'_w$,
\[
\E^\circ_{T'_w}g
=
\E^\circ_{T'_w}\E^\circ_R f
=
\E^\circ_{Q_w}\E^\circ_{T'_w}f .
\]
The set $Q_w$ is at tree distance at least $2$ from $T_w$. Applying \eqref{eq:var_convexity} with
$A=T_w$ and $B=Q_w$, we get for any $\tau\in \Omega$
\[
\Var^{\tau}_{T_w}
\left(\E^\circ_{T'_w}g\right)
=
\Var^{\tau}_{T_w}
\left(\E^\circ_{Q_w}\E^\circ_{T'_w}f \right)
\leq
\E^\tau_{Q_w}
\left[
\Var^{\circ}_{T_w}
\left(\E^\circ_{T'_w}f\right)
\right].
\]
Taking expectation over $\tau\sim \nu^\eta_{T_v}$ and using that $\E^\eta_{T_v}\E^\circ_{Q_w}=\E^\eta_{T_v}$, this yields
\[
\E^\eta_{T_v}
\left[
\Var^{\circ}_{T_w}
\left(\E^\circ_{T'_w}g\right)
\right]
\leq
\E^\eta_{T_v}
\left[
\Var^{\circ}_{T_w}
\left(E^\circ_{T'_w}f\right)
\right].
\]
Summing over $w\in B_v'\cup \partial B_v$ and combining with \eqref{eq:infefer} establishes \eqref{eq:goalrf4f4f}, completing therefore the proof.
\end{proof}

\subsection{Proof of Lemma~\ref{lem:spectologSob}}\label{sec:lemspectologSob}
\begin{lemspectologSob}
\statelemspectologSob
\end{lemspectologSob}

\begin{proof}
	Let $v \in V(T)$ and consider a feasible boundary condition $\tau$ outside
	$T_v$. Let $\alpha(P^\tau_{T_v})$ and $\gamma(P^\tau_{T_v})$ be the log-Sobolev constant
	and spectral gap of the discrete-time Glauber dynamics for $\nu^\tau_{T_v}$.
	Thus, for any $g \colon \Omega\to\mathbb R$,
	\[\mathcal{E}_{T_v}^\tau(g) \coloneqq \frac{1}{|T_v|} \sum_{w \in T_v}\ \E^\tau_{T_v} \left[\Var^\circ_w (g)\right],
    \hspace{1.5em}
    \gamma(P^\tau_{T_v}) \coloneqq \inf_g \frac{\mathcal{E}_{T_v}^\tau(g)}{\Var^\tau_{T_v} (g)},
	\hspace{1.5em}
    \alpha(P^\tau_{T_v}) \coloneqq \inf_{f \geq 0} \frac{\mathcal{E}_{T_v}^\tau(\sqrt f)}{\Ent^\tau_{T_v}(f)}.\]
	
	Fix a non-negative function $f$. By the entropy variant of (\ref{eq:var_total_law}), we have that
	\begin{equation} \label{eq:gsc1}
		\Ent^\tau_{T_v} (f) = \E^\tau_{T_v} \left[\Ent^\circ_{T'_v} (f)\right] + \Ent^\tau_{T_v} \left[\E^\circ_{T'_v} (f)\right].
	\end{equation}
	
	We begin by upper-bounding the first term on the r.h.s of (\ref{eq:gsc1}). We use the notation $u \prec v$ to denote that $u$ is a child of $v$. Conditioned on the spin at $v$, the Gibbs distribution $\nu$ on $T'_v$ is product over those on child subtrees $T_u$ for $u \prec v$. Therefore, using the product-measure entropy inequality, we get
	\begin{equation*}
		\E^\tau_{T_v} \left[\Ent^\circ_{T'_v} (f)\right] \leq \sum_{u \prec v} \E^\tau_{T_v} \left[\Ent^\circ_{T_u} (f)\right] 
	\end{equation*}	

	Applying the log-Sobolev inequality inside each $T_u$ yields
	\begin{align}
		\E^\tau_{T_v} \left[\Ent^\circ_{T'_v} (f)\right] &\leq \sum_{u \prec v} \E^\tau_{T_v} \left[\frac{1}{|T_u| \alpha(P^\circ_{T_u})} \sum_{w \in T_u} \E^\circ_{T_u} \left[\Var^\circ_w (\sqrt{f})\right]\right] \nonumber \\
		&\leq \sum_{u \prec v}  \left(\max_{\eta}\frac{1}{|T_u| \alpha(P^\eta_{T_u})}\right) \sum_{w \in T_u} \E^\tau_{T_v} \left[\E^\circ_{T_u} \left[\Var^\circ_w (\sqrt{f})\right]\right] \nonumber \\
		&= \sum_{u \prec v}  \left(\max_{\eta}\frac{1}{|T_u| \alpha(P^\eta_{T_u})}\right) \sum_{w \in T_u} \E^\tau_{T_v} \left[\Var^\circ_w (\sqrt{f})\right] \nonumber \\
		&\leq \left(\max_{u \prec v, \eta}\frac{1}{|T_u| \alpha(P^\eta_{T_u})}\right) \sum_{w \in T'_v} \E^\tau_{T_v} \left[\Var^\circ_w (\sqrt{f})\right] \label{eq:gsc4}
	\end{align}
	where the third equality follows from the fact that for any $\tau \in \Omega$ and $g \coloneqq \Omega \to \mathbb{R}$, if $B \subseteq A$ then $\E^\tau_A \E_B^{\circ} (f) = \E^\tau_A (f)$.
	
	We next upper-bound the second term on the r.h.s of (\ref{eq:gsc1}). Since $\nu$ is $b$-marginally bounded, every positive marginal probability of $\nu^\tau_{T_v}$ at $v$ is at least $b$. Let $C(b) \coloneqq \frac{\log((1/b) - 1)}{1 - 2b}$. From Lemma \ref{lem:mixingtimebounds}, noting that the modified log-Sobolev of a one-site chain is $1$, we have that the log-Sobolev inequality with function $g = \E^\circ_{T'_v} (f)$ gives
	\begin{equation} \label{eq:gsc2}
		\Ent^\tau_{T_v} \left[\E^\circ_{T'_v} (f)\right] \leq C(b) \Var^\tau_{T_v} \left[\sqrt{\E^\circ_{T'_v} (f)}\right] \leq C(b) \Var^\tau_{T_v} \left(\sqrt{f}\right).
	\end{equation}
	
	By the spectral-gap inequality for $P^\tau_{T_v}$,
	\begin{equation} \label{eq:gsc3}
		\Var^\tau_{T_v} (\sqrt{f}) \leq \frac{1}{|T_v| \gamma(P^\tau_{T_v})} \sum_{w \in T_v} \E^\tau_{T_v} \left[\Var^\circ_w (\sqrt{f})\right].
	\end{equation}
	
	Combining (\ref{eq:gsc2}) and (\ref{eq:gsc3}) together, we obtain
	\begin{equation} \label{eq:gsc5}
			\Ent^\tau_{T_v} \left[\E^\circ_{T'_v} (f)\right] \leq \frac{C(b)}{|T_v| \gamma(P^\tau_{T_v})} \sum_{w \in T_v} \E^\tau_{T_v} \left[\Var^\circ_w (\sqrt{f})\right].
	\end{equation}
	
	Consequently, from (\ref{eq:gsc1}), (\ref{eq:gsc4}) and (\ref{eq:gsc5}) we obtain that
	\begin{equation} \label{eq:gsc6}
		\Ent^\tau_{T_v} (f) \leq \left(\frac{C(b)}{|T_v| \gamma(P^\tau_{T_v})} + \max_{u \prec v, \eta}\frac{1}{|T_u| \alpha(P^\eta_{T_u})}\right) \sum_{w \in T_v} \E^\tau_{T_v} \left[\Var^\circ_w (\sqrt{f})\right].
	\end{equation}
	
	Using the definition of $\alpha(P^\tau_{T_v})$ and taking the infimum over all non-negative functions $f$, from (\ref{eq:gsc6}) we get
	\begin{equation*}
		\frac{1}{|T_v| \alpha(P^\tau_{T_v})} \leq \frac{C(b)}{|T_v| \gamma(P^\tau_{T_v})} + \max_{u \prec v, \eta}\frac{1}{|T_u| \alpha(P^\eta_{T_u})} 
	\end{equation*}
	and recursing along every root-to-leaf path from the root of $T$ yields
	\begin{equation*}
		\frac{1}{|T| \alpha} \leq \frac{(h + 1) C(b)}{\min_{v, \eta} |T_v| \gamma(P^\eta_{T_v})}.
	\end{equation*}
	
	Re-arranging and substituting for $\gamma^*$, we get $\alpha \geq \frac{1}{2h}  \frac{1-2b}{\log ((1/b)-1)} \gamma^*$, as required.
\end{proof}

\section{Structural Properties of Poisson Trees and Random Graphs}
Throughout this section, we fix a real number~$\epsilon \in (0,1)$ and a 
positive integer~$d$ that is sufficiently large with respect to~$\epsilon$. Let $D \coloneqq \lceil (1-\epsilon) d \rceil$.
Let $\cT$ be the distribution of the infinite Poisson tree with parameter~$d$.
\subsection{Poisson Trees: Proof of Lemma~\ref{lem:degree-sum}}\label{sec:tree-properties}
To facilitate the proof of Lemma~\ref{lem:degree-sum}, 
we define  a larger subtree of the infinite Poisson tree~$\cT$ which contains $T^h_{v,\ell}$ and 
prove a version of Lemma~\ref{lem:degree-sum} for this larger subtree.
Given an infinite tree~$T$ and a vertex $v\in V(T)$,
$T_v$ denotes the subtree of~$T$ rooted at~$v$.
We start with a definition that is analogous to Definition~\ref{def:good-bad-vertices}
but without reference to the truncation parameter~$h$.

\begin{definition}\label{def:good-inf} Let $T \sim \cT$.
A vertex~$v$ of $T$ is said to be \emph{good} if  
$T_v$ has an infinite $D$-ary subtree rooted at~$v$.
Otherwise, $v$ is said to be \emph{bad}.
\end{definition}

\begin{definition}\label{def:pgood} 
Let $p_\good$ be   the largest positive solution to $p_\good = \Pr(\Poisson(dp_\good) \geq D)$  and let $p_\bad \coloneqq 1 - p_\good$.
\end{definition}

\begin{lemma}[{\cite[Theorem 5.29]{Lyons_Peres_2017}}]\label{lem:pgood-exists}
$p_\good$ is well-defined and is in $[0,1]$. If $T\sim \cT$ then 
the probability that any $v\in V(T)$ is good is $p_\good$.
\end{lemma}

\begin{definition}\label{def:inf-Tvl}
Let $\ell$ be a positive integer and
let $T\sim \cT$.  For every vertex $v\in V(T)$, 
let $T_{v,\ell}$ be the maximal subtree of~$T_v$ in which each root-to-leaf path contains at most $\ell$ good vertices, excluding the root~$v$.
\end{definition}

    

Lemma~\ref{lem:goodvarious} justifies Definition~\ref{def:good-inf}.

\begin{lemma}\label{lem:goodvarious}
Fix a positive integer~$h$.
Let $T\sim \cT$. Let $v$ be a vertex at depth at most~$h$ in~$T$.
If $v$ is good (Definition~\ref{def:good-inf}) then it is $h$-good (Definition~\ref{def:good-bad-vertices}).
\end{lemma}
\begin{proof}
If $v$ is good then $T_v$ has an infinite $D$-ary subtree rooted at~$v$. Every vertex in this subtree is good. Let $h'\leq h$ be the depth of~$v$ in $T$.
Let $\ell = h-h'$. Then $T$ contains a complete $D$-ary subtree~$\tree$ of depth~$\ell$ rooted at~$v$. By Definition~\ref{def:good-bad-vertices}, the leaves of~$\tree$ are $h$-good, and so are their parents. 
Applying the definition from the leaves of $\tree$ up to~$v$, all nodes in $\tree$ (including $v$) are $h$-good.
 \end{proof}

\begin{lemma}\label{lem:domTvl}
Let~$h,\ell\geq 1$ be integers and  $T\sim \cT$. For every $v\in V(T^h)$, $T^h_{v,\ell}$ is a subtree of $T_{v,\ell}$.
\end{lemma}
\begin{proof}
Let $w\in V(T^h_{v,\ell})$. By Definition~\ref{def:level}, the path from
$v$ to $w$, excluding $v$, contains at most $\ell$ vertices that are $h$-good.
By Lemma~\ref{lem:goodvarious}, every good vertex on this path is $h$-good so the path
contains at most $\ell$ good vertices. By Definition~\ref{def:inf-Tvl}, $w\in V(T_{v,\ell})$.
\end{proof}

The trees are easier to work with than the trees~$T^h_{v,\ell}$ in Lemma~\ref{lem:degree-sum} but they are still quite challenging to analyse, since determining whether a vertex of~$\cT$  is good (Definition~\ref{def:good-inf}) requires exposing the whole tree. To make things more straightforward, we define a new process $\cT^{\down}$ (Definition~\ref{def:two-type-equivalent}) that generates an infinite rooted tree whose vertices are \emph{down-good} and \emph{down-bad}; we then show (Lemma~\ref{lem:gb}) that the new tree has the same distribution as~$\cT$ and
even (Lemma~\ref{lem:downgood}) that every down-good vertex is good, thus, we will be able to work with the new process $\cT^{\down}$ (or actually, with another ``downward'' process that dominates it). 

\begin{definition}\label{def:two-type-equivalent}
Let $\Downg$  be the distribution of  $X\sim\Poisson(dp_\good)$.
Let $\Downgbig$ be the distribution of $X$ conditioned on $X\geq D$ and let $\Downgsmall$ be $\Downg$ conditioned on $X<D$. Let $\Downb$ be the distribution $\Poisson(dp_\bad)$.  The rooted random tree $\cT^{\down}$ is defined as follows.

\begin{itemize}
\item The root is set to be \emph{down-good} with probability $p_\good$ and \emph{down-bad} otherwise (with probability $p_\bad$).

\item From every down-good vertex, the number of down-good children is sampled independently from $\Downgbig$ and the number of down-bad children is sampled independently from $\Downb$.

\item From every down-bad vertex, the number of down-good children is sampled independently from $\Downgsmall$ and the number of down-bad children is sampled independently from $\Downb$.
\end{itemize}

\end{definition}

Lemma~\ref{lem:gb} shows that the vertices of $\cT^{\down}$ are in one-to-one correspondence with the vertices of~$\cT$.

\begin{lemma}\label{lem:gb}
In $\cT^{\down}$, every node~$v$ has probability $p_\good$ of being down-good and its number of children is a Poisson random variable with parameter~$d$.
\end{lemma} 
\begin{proof}
 The proof is by induction on the depth of~$v$.
For the base case, the root has depth~$0$ and it is clear by construction that it has probability $p_\good$ of being down-good. 

Consider a node $v$ that has probability $p_\good$ of being down-good. We now consider its descendants. With probability $p_\good$, it is down-good and the number of down-good children is chosen from $\Downgbig$
whereas with probability $1-p_\good$ it is down-bad and the number of down-good children is chosen from $\Downgsmall$.
Recall from the definition of~$p_\good$ (Definition~\ref{def:pgood}) that $p_\good$ is
the probability that an output of $\Downg$ is at least~$D$. Thus, the number of its down-good children has distribution $\Downg$, as
\begin{align*}
    p_\good \cdot \Pr(\Poisson(dp_\good)=\cdot \mid \Poisson(dp_\good)\ge D) &+ p_{\bad}\cdot \Pr(\Poisson(dp_\good)=\cdot \mid \Poisson(dp_\good)< D)\\
&= \Pr(\Poisson(dp_\good)=\cdot ).
\end{align*}
Since the number of its down-bad children has distribution $\Downb$, its total number of children is a Poisson random variable with parameter~$d$. Also, the probability that any child is down-good is $p_\good$.
\end{proof}

\begin{lemma}
\label{lem:downgood}
Consider $T\sim \cT$. Using the correspondence in Lemma~\ref{lem:gb}, $T\sim \cT^{\down}$.
If a vertex $v\in V(T)$ is down-good then it is good.
\end{lemma}
\begin{proof}
If $v$ is down-good then it has at least~$D$ down-good children, and all of its descendants have a least~$D$ down-good children, so $T_v$ has an infinite $D$-ary subtree rooted at~$v$. Hence, $v$ is good (Definition~\ref{def:good-inf}).
\end{proof}

\begin{definition}\label{def:down-Tvl}
Let $\ell$ be a positive integer and
let $T\sim \cT^{\down}$.  For every vertex $v\in V(T)$, 
let $T^{\down}_{v,\ell}$ be the maximal subtree of~$T_v$ in which each root-to-leaf path contains at most $\ell$ down-good vertices, excluding the root~$v$.
\end{definition}

Corollary~\ref{cor:downgood} follows immediately from Lemma~\ref{lem:downgood}.
\begin{corollary}
\label{cor:downgood}   
Let $\ell$ be a positive integer and
consider $T\sim \cT$. Using the correspondence in Lemma~\ref{lem:gb}, $T\sim \cT^{\down}$. For every vertex $v\in V(T)$,  $T_{v,\ell}$ from Definition~\ref{def:inf-Tvl}
is a subtree of $T^{\down}_{v,\ell}$.

\end{corollary}

The tree $T^{\down}_{v,\ell}$ is a much easier object in order to work with (in the context of proving Lemma~\ref{lem:degree-sum}). To bound $|V(T_{v,\ell}^{\down})|$ 
we introduce one final domination.  The idea is to \textit{compress} the down-bad vertices: given a vertex $v$, if vertex $w$ is a down-good descendant of $v$ such that   all vertices on the path from $v$ to $w$ (excluding the endpoints) are down-bad, then we say that $w$ is an \textit{indirect (down-good) offspring} of $v$.
The direct offspring of $v$ are simply its down-good children. We now define a process $\cT^{\dom}$ that dominates the number of direct and indirect offspring  in $\cT^{\down}$.

Before defining $\cT^{\dom}$, we give Lemma~\ref{lem:pbad-small}, which gives a useful lower bound for $p_\good$. It implies that the Galton--Watson branching process in which the number of descendants is a Poisson random variable with parameter $d p_\bad$ is subcritical.

\begin{lemma}\label{lem:pbad-small}
    Let $\epsilon \in (0,1)$. Then for all sufficiently large $d>1$,  $p_\good \geq 1 - \emm^{-d\epsilon^2/4}$.
\end{lemma}
\begin{proof} 
    Fix $\epsilon \in (0, 1)$ and assume $d$ is large enough 
    so that $\exp\left(-\frac{d\epsilon^2}{4}\right) \leq \frac\epsilon 2$.

    Let $f(x) \coloneqq \Pr(\Poisson(dx) \geq D)$ over $x \in [0, 1]$. By Lemma~\ref{lem:pgood-exists}, 
    $p_\good$ is the largest fixed point of $x = f(x)$. Observe that $f(1) < 1$. Then, because $f$ is continuous, if for some $x_0\in[0,1]$, $x_0\leq f(x_0)$ then there must exist solution to $x = f(x)$ in $[x_0,1]$. Since $p_\good$ is the largest solution to $x = f(x)$ in $[0,1]$, it suffices to show
    \begin{equation} \label{eq:pgood_bdd_1}
        x_0 \coloneqq 1 - \emm^{-d\epsilon^2/4} \leq \Pr(\Poisson(dx_0)\geq D) .
    \end{equation}   
    By the choice of $d$, we have $x_0 \geq 1 - \epsilon/2$. By a Chernoff bound, we get
    \[
    \Pr(\Poisson(d x_0) \geq D) \geq 1 - \exp\left(-\frac{(d x_0 - d(1-\epsilon))^2}{2 d x_0 - d(1-\epsilon)}\right) \geq 1 - \exp\left\{-\frac{(d\epsilon/2)^2}{d(2 - \epsilon - (1-\epsilon))}
    \right\}.
    \]
    This is at least $1 - \emm^{-d\epsilon^2/4}$, which concludes the proof of \eqref{eq:pgood_bdd_1}.
\end{proof}

\begin{definition}\label{def:compressed}
Let \(\cS_\bad\) be the distribution of the size of a (subcritical)  
    \(\Poisson(dp_\bad)\) Galton--Watson branching process.
Let $\cT^\dom$ be the Galton--Watson branching process with  offspring distribution \(\bm X\), where 
\[
    \bm X \coloneqq D + X_\good + D\sum_{i=1}^{X_\bad} S_i,
    \]
where $X_\good\sim\Poisson(dp_\good)$, $X_\bad\sim\Poisson(dp_\bad)$, and
$\{S_i\}_{i\geq 1}$ are independent, with $S_i\sim \cS_\bad$ for each $i$.
For every non-negative integer~$\ell$, let 
$\cT^\dom_\ell$ be the subtree of~$\cT^{\dom}$ containing all nodes of depth at most~$\ell$. 
\end{definition}

The process~$\cT^{\dom}$ can be used to dominate 
the number of vertices in $T^{\down}_{v,\ell}$,  as shown in  
Lemmas~\ref{lem:good-dominated} and~\ref{lem:bad-dominated}.

\begin{lemma}\label{lem:good-dominated}
Let $d$ be a sufficiently large real and let  $T\sim \cT^{\down}$. For any $v\in V(T)$ and $\ell > 0$, the number of down-good vertices in $\cT^{\down}_{v,\ell}$ is stochastically dominated by the number of vertices in $\cT^\dom_\ell$.
\end{lemma}

\begin{proof}
Fix $v\in V(T)$. By Definition~\ref{def:two-type-equivalent}, $v$ has $\Poisson(dp_\bad)$ bad children. By Lemma~\ref{lem:pbad-small}, for all $d$ large enough, $d p_\bad = \emm^{-\Omega(d)}$ wich is much smaller than~$1$. As a result, the maximum connected subtrees of bad vertices rooted in the bad children of $v$ are distributed as the subcritical $\Poisson(dp_\bad)$ Galton--Watson process. By Definition~\ref{def:two-type-equivalent}, every bad vertex has at most $D-1$ good offspring, thus, the number of indirect good offspring of $v$ is dominated by  $D\sum_{i=1}^{X_\bad} S_i,$     where $X_\bad$ and $ S_i$ are as in Definition~\ref{def:compressed}. Then it suffices to show that $D + X_\good \sim D + \Poisson(dp_\good)$ dominates the number of direct good offspring of $v$, which is $\Poisson(dp_\good)$ conditioned on the outcome being at least~$D$.

    If $v$ is bad, then the number of good offspring of $v$ in $T$ does not exceed $D$, giving us the domination. If $v$ is good, it is enough to show that for $\bm g\sim \Poisson(dp_\good)$, $X_\good\sim\Poisson(dp_\good)$ and integer $x\geq 0$ it holds that 
    \begin{equation}\label{eq:yh57h5g}
    \Pr(\bm g \geq x \mid \bm g \geq D) \leq \Pr(X_\good + D \geq x).
    \end{equation}
    Indeed, the result is trivial when $D\geq x$, so assume $D<x$. Let $X\sim\mathrm{Poi}(\lambda)$ where $\lambda=dp_{\mathrm{good}}$ and consider a Poisson process $N(t)$ with rate $\lambda$. Then $X$ has the same distribution as $N(1)$. For integer $m\geq 0$, let $T_m$ be the time of the $m$-th arrival. Then $\Pr[X\geq x]=\Pr[T_x\leq 1]$. But $T_x$ has the same distribution as $T_{x-D}+T_D'$ (where $T_D'$ is an independent copy of $T_D$) and
    \[\Pr[T_{x-D}+T_D'\leq 1]\leq \Pr[T_{x-D}, T_D'\leq 1]=\Pr[T_{x-D}\leq 1]\, \Pr[T_D'\leq 1]=\Pr[X\geq x-D]\,\Pr[X\geq D],\]
    proving \eqref{eq:yh57h5g}.
\end{proof}

We also get the following domination result for the number of down-bad vertices.

\begin{lemma}\label{lem:bad-dominated}
Let $d>1$ be a sufficiently large real and let  $T\sim \cT^{\down}$. For any $v\in V(T)$ and $\ell > 0$, the number of down-bad vertices in $\cT^{\down}_{v,\ell}$ is stochastically dominated by the number of vertices in $\cT^\dom_{\ell+1}$.
\end{lemma}
\begin{proof}
Consider any vertex $w\in V(T)$ and the number of  maximal connected subtrees consisting of down-bad vertices rooted in (down-bad) children of $w$. The size of each such subtree is distributed as the size of subcritical $\Poisson(dp_\bad)$ Galton--Watson process -- $\cS_\bad$. Since there are $X_\bad\sim\Poisson(dp_\bad)$ bad children, the number of bad vertices that would be ``compressed'' onto $w$ is distributed is $\sum_{i=1}^{X_\bad} S_i$, where $S_i$ are i.i.d. from $\cS_\bad$.

In the compressed tree $T^\dom$ we get $D \geq 1$ offsprings for each of the bad vertices, thus the number of bad vertices in $T_{v,\ell}^{\down}$ is dominated by the number of  vertices in $\cT^\dom_{\ell+1}$, where the extra level is to account for the fact that $T^{\down}_{v,\ell}$ also includes the down-bad vertices adjacent to the $\ell$-th layer of down-good vertices.
\end{proof}

Next we prove that, taking $\ell =  \Theta(\log \log n)$ 
the size of $\cT^{\dom}_\ell$ is 
at most a polynomial in $\log n$ with probability at least~$1/\poly(n)$. To do this, we first note the probability generating function of the offspring distribution of $\cT^{\dom}$ and we then use tail bounds.

\begin{lemma}\label{lem:pgf-of-X}
Let $d>1$ be a sufficiently large real so that $dp_\bad < 1$.
    Then, the probability generating function $f(t) \coloneqq \E[t^{\bm X}]$ of \(\bm X\), the offspring distribution on $\cT^\dom$, is well defined for \(0 \leq t \leq (\frac{\emm^{dp_\bad}-1}{dp_\bad})^{1/D}\), 
    and for such $t$ we have 
    \[
    f(t) = g(t^D) \emm^{dp_\good(t-1)},
    \]
    where $g(t) \coloneqq \E[t^{\cS_\bad}]$ is the least non-negative solution $z_0$ of $z = t \emm^{dp_\bad(z-1)}$, 
    which is well-defined
    for $|t|\leq \tfrac{\emm^{dp_\bad - 1}}{dp_\bad}$.
\end{lemma}
\begin{proof}
    It is a standard result (see for example \cite[Theorem~3.16]{vdHofstad2017}) that the distribution $S_\bad$  is the Borel distribution; its probability generating function $g(t)$ is defined for $|t|\leq \tfrac{\emm^{dp_\bad - 1}}{dp_\bad}$ as the least non-negative solution to $z = t\emm^{dp_\bad(z - 1)}$. 
    Thus, when  $0\leq t\leq (\tfrac{\emm^{dp_\bad - 1}}{dp_\bad})^{1/D}$, using the definition of $\bm X$ and that the probability generating function of a $\Poisson(\lambda)$ random variable is $\emm^{\lambda(t-1)}$, we get that
    \begin{align*}
    \E[t^{\bm X}] &= \E[t^{D+X_\good + D\sum_{i=1}^{X_\bad} \sum_i S_i}] = t^D \emm^{dp_\good(t - 1)} \E\left[\E\left[\prod_i^{X_\bad} (t^{D})^{S_i}\mid X_\bad\right]\right]  \\
    &= t^{D}\emm^{dp_\good(t-1)} \E[g(t^D)^{X_\bad}] = t^{D}\emm^{dp_\good(t-1) + dp_\bad(g(t^D) - 1)}.
    \end{align*}
    By noting $\emm^{dp_\bad(g(t^D) - 1)} = g(t^D)/t^D$, we get that $
    f(t) = g(t^D) \emm^{dp_\good(t-1)}$.
\end{proof}

Next, we get the probability generating function for $|\cT^\dom_\ell|$, the size of the branching process with offspring distribution \(\bm X\) after $\ell$ generations.

\begin{lemma}\label{lem:process-size-pgf}
    For $\ell \geq 0$, let $Z_\ell \coloneqq |\cT^\dom_\ell|$ be the total number of vertices in the branching process after $\ell$ generations. Then, \(Z_\ell\) has a probability generating function, defined recursively by
    \[
    f_\ell(t) \coloneqq \E[t^{Z_\ell}] = \begin{cases}
        t & \text{if } \ell = 0 \\
        t f(f_{\ell-1}(t)) & \text{otherwise}
    \end{cases}
    \]
    The function \(f_0\) is defined for all \(t\geq 0\), and \(f_\ell\) is defined for all \(t\geq 0\) small enough so that \(f_{\ell-1}(t)\leq (\frac{\emm^{dp_\bad-1}}{dp_\bad})^{1/D}\).
\end{lemma}
\begin{proof}
    We prove this by induction. For $\ell = 0$, the tree $\cT_0^\dom$ consists solely of the root, thus $Z_0 = 1$, giving $f_0(t) = t$. For $\ell > 0$, let $\bm X_0$ be the number of children of the root, and let $T_1,\dots, T_{\bm X_0}$ be independent copies of $\cT^\dom_{\ell-1}$ -- the subtrees of the children of the root. By conditioning on $\bm X_0$ we get
    \begin{align*}
        f_\ell(t) = \E\left[\E\left[t^{1 + \sum_{i=0}^{\bm X_0} |T_i|} \mid \bm X_0\right]\right] = t\E[f_{\ell-1}(t)^{\bm X_0}] = tf(f_{\ell-1}(t)).
    \end{align*}
    Since $\bm{X} > 0$ and $f(t)=\E[t^{\bm{X}}]$, $f(t)$ is an increasing function in $t$. 
    The function $f_\ell(t)$ is defined for $t\geq 0$ where $f(f_{\ell-1}(t))$ is defined, which is for $t$ small enough  satisfying $|f_{\ell-1}(t)|\leq(\frac{\emm^{dp_\bad-1}}{dp_\bad})^{1/D}$.
\end{proof}

Now we use the probability generating function to prove a tail bound for the size of $Z_{\ell}$.
\begin{lemma} \label{lem:tgood_pgf_bdd}
    For all sufficiently large \(d\), for \(\ell \geq 0\) and all \(\theta \leq1\), \(f_\ell(\emm^{\theta / (6d)^\ell})\) is defined and moreover \(f_\ell(\emm^{\theta / (6d)^\ell}) \leq \emm^{\theta}\).
\end{lemma}

\begin{proof}
    We proceed by induction on \(\ell\). For \(\ell = 0\), \(f_0(t)\) is defined for all \(t\), and in particular \(f_0(\emm^{\theta/1})=\emm^\theta\) for all \(\theta\leq 1\). Assume that the lemma is true for \(\ell\geq 0\) and fix any $\theta \leq 1$. Using Lemma~\ref{lem:process-size-pgf},
\[
    f_{\ell+1}(\emm^{\theta/(6d)^{\ell+1}}) = \emm^{\theta /(6d)^{\ell+1}} f(f_\ell(\emm^{\theta/(6d)^{\ell+1}})).
    \]

    First we show that this is defined for all $ 0 \leq \theta \leq 1$. 
    To see this by Lemma~\ref{lem:process-size-pgf}, we need $f_\ell(\emm^{\theta/(6d)^{\ell+1}}) \leq(\frac{\emm^{dp_\bad-1}}{dp_\bad})^{1/D}$.
    Note that \(\frac{\theta}{(6d)^{\ell+1}} = \frac{\theta/(6d)}{(6d)^\ell}\) and \(\theta/(6d) < \theta\leq 1\). Thus, by induction hypothesis on \(f_\ell\) with \(\theta' = \theta/(6d)\), 
    it suffices to show 
    $\emm^{\theta'} \le (\frac{\emm^{dp_\bad-1}}{dp_\bad})^{1/D}$.
    Then using \(D \leq d\) and \(\theta \leq 1\), we get that
    \(
    \emm^{D\theta'} =
    \emm^{\theta D/(6d)} \leq \emm^{1/6} <1 \). 
    Thus, $f_{\ell+1}(\emm^{\theta/(6d)^{\ell+1}})$ is defined.
    
    Next, by noting that \(f\) is an increasing function, using that $f_l(\theta'/(6d)^l) \le \emm^{\theta'}$, 
 we obtain

    \begin{equation}
        \label{eq:induction-bound-on-fl+1}
    f_{\ell+1}(\emm^{\theta/(6d)^{\ell+1}})
    \leq \emm^{\theta /(6d)^{\ell+1}} f(\emm^{\theta'} )
    \leq\emm^{\theta
    /(6d)^{\ell+1}+dp_\good (\emm^{\theta/(6d)} - 1)} g(\emm^{\theta D/(6d)}) .
    \end{equation}

    Assume from now $d$ is sufficiently large. In the next step, we use the following claim.

{\bf Claim: }  
        For all $z\leq 1$, $g(\emm^z)\leq \emm^{2z}$.
        
    Assuming the claim, we get that 
    \(g(\emm^{\theta D/(6d)})\leq \emm^{\theta D/(3d)}\leq \emm^{\theta / 3}\). Plugging back into~\eqref{eq:induction-bound-on-fl+1}, we obtain that
    \begin{align*}
        f_{\ell+1}(\emm^{\theta/(6d)^{\ell+1}})\leq& \emm^{\theta (6d)^{-\ell - 1}+d(\emm^{\theta/(6d)} - 1) + \theta /3} \\
        \leq & \exp\{\theta[1/(6d)^{(\ell+1)} + d(\emm^{\theta/(6d)}-1)/\theta + 1/3]\}.
    \end{align*}
     Finally, using \(D\leq d\) and that for \(x\in(0,1)\) \(\emm^{x}-1\leq 2x\), we conclude
    \[
    f_{\ell+1}(\emm^{\theta/(6d)^{\ell+1}})\leq \exp\{\theta[1/(6d)+1/3+1/3]\} \leq \emm^{\theta}.
    \]
{\bf Proof of the Claim: } 
    Recall that $g(t)$ is the probability generating function of the total
progeny of a subcritical $\operatorname{Poi}(\lambda)$ Galton--Watson process where $\lambda =d p_{\mathrm{bad}}$, and hence
is the smallest positive solution of $    g(t)=t\emm^{\lambda(g(t)-1)}$. Let $t=\emm^z$ and $x=\emm^{2z}$, so that our goal is to show that $g(t)\leq x$, or equivalently that $x\ge t\emm^{\lambda(x-1)}$.
Substituting $t=\emm^z$ and $x=\emm^{2z}$, this becomes $z\ge \lambda(e^{2z}-1)$.

For $0\le z\le1$ we have $\emm^{2z}-1\leq (\emm^2-1)z$.  Therefore it suffices to have $\lambda\leq (\emm^2-1)^{-1}$.  Since
$\lambda=d p_{\rm bad}=\emm^{-\Omega(d)}$, this holds for all sufficiently
large $d$. This ends the proof of the claim. 
\end{proof}

\begin{lemma}
\label{lem:Tvl-polylog} 
Fix arbitrary  reals $\kappa,C_0>0$. Then, for all large enough~$n$ and 
$\ell = \lfloor C_0\log\log n\rfloor $, for $T\sim \cT^{\down}$ and every $v\in V(T)$, it holds that
$\Pr\big( |T^{\down}_{v,\ell}| \leq (6d)^\ell (\log n)^2\big) \geq 1-1/n^{\kappa+1}$. 
\end{lemma}
\begin{proof}   
By Lemma~\ref{lem:good-dominated}
the number of down-good vertices in $T^{\down}_{v,\ell}$ is stochastically dominated by the number of vertices in $\cT^{\dom}_\ell$ 
and by Lemma~\ref{lem:bad-dominated}, 
the number of down-bad vertices in $\cT^{\down}_{v,\ell}$ is stochastically dominated by the number of vertices in $\cT^{\dom}_{\ell+1}$.
From Lemma~\ref{lem:process-size-pgf}, $Z_\ell = |\cT^{\dom}_\ell|$.
Hence, by a union bound and the monotonicity $Z_\ell\le Z_{\ell+1}$, it is enough to show that
$\Pr\left(Z_{\ell+1}\ge \frac12(6d)^\ell(\log n)^2\right)
   \leq \frac{1}{2n^{\kappa+1}}$. By Markov's inequality,
\[  
\Pr(Z_{\ell+1}  \geq \tfrac{1}{2}(6d)^\ell (\log n)^2) = \Pr(\emm^{Z_{\ell+1} / (6d)^{\ell+1}}\geq \emm^{(\log n)^2/(12d)})\leq \frac{\E[\emm^{Z_{\ell+1}/(6d)^{\ell+1}}]}{n^{(\log n) / (12d)}} \leq \emm\cdot n^{-\tfrac{\log n}{12d}},
\]  
where the last inequality follows from Lemma~\ref{lem:tgood_pgf_bdd} (with $\theta = 1$). For large enough $n$, the probability bound is less than $1/(2n^{\kappa+1})$ as needed.
\end{proof}

Finally, we conclude the section by proving Lemma~\ref{lem:degree-sum}, which we restate for convenience. 
\begin{lemdegreesum}
\statelemdegreesum
\end{lemdegreesum}

\begin{proof}
Fix the constants as in the statement of the lemma and take $C_1$ sufficiently large, to be determined below.
Let $n$ be sufficiently large.  By Lemma~\ref{lem:domTvl}, we have  that $T^{h}_{v,\ell}$ is a subtree of $T_{v,\ell}$ and, by Corollary~\ref{cor:downgood}, $T_{v,\ell}$ is a subtree of $T^{\down}_{v,\ell}$. So, it suffices to lower-bound the probability of the event $\mathcal{E}(T^{\down}_{v,\ell})$ for a fixed $v\in V(\cT^{\down})$. 
 
By applying Lemma~\ref{lem:Tvl-polylog} with $\kappa' > \kappa+K+1$,  for all sufficiently large~$n$, $\Pr(|T^{\down}_{v,\ell}|\leq N)\geq 1 - n^{-(\kappa + 2)} d^{-h}$ for $N=  (6d)^\ell (\log n)^2$, which is at most a polynomial in $\log n$. 
 
Consider generating $\cT_{v,\ell}^{\down}$ as follows:
by a breadth-first-search, with degrees being drawn from two independent sequences of i.i.d random variables pairs: $(\bm g^\good_i, \bm b^\good_i)_{i\geq 1}$ for good vertices 
from the distribution $\Downgbig \times \Downb$ 
as in Definition~\ref{def:two-type-equivalent}
and $(\bm g^\bad_i, \bm b^\bad_i)_{i\geq 1}$
for bad vertices
from the distribution $\Downgsmall \times \Downb$. As we showed in the proof of Lemma~\ref{lem:good-dominated}, both $\bm g^\good_i + \bm b^\good_i$ and $\bm g^\bad_i + \bm b^\bad_i$ are dominated by $\Poisson(d) + D$.

Let $\bm d_1,\dots, \bm d_{2N}$ be an i.i.d. sequence of $\Poisson(d)+D$ random variables such that, for all $i\leq N$, $\bm d_i \geq \bm g^\good_i + \bm b^\good_i$ and, for all $N < i\leq 2N$, $\bm d_i \geq \bm g^\bad_{i-N} + \bm b^\bad_{i-N}$.
Then, conditioned on $|T^{\down}_{v,\ell}|\leq N$, the sequence of degrees in $\cT^{\down}_{v,\ell}$ is contained, in some order, in $(\bm d_i)_{1\leq i\leq 2N}$, as $\cT^{\down}_{v,\ell}$ then contains at most $N$ bad and $N$ good vertices. 

Let $n' \coloneqq \log_2 (2N)$ and $W \coloneqq \frac{\log n}{\log \log n}$.
Since $n$ is sufficiently large,   $N\geq n'$. Therefore, whenever there is a set $S\subseteq V(\cT^{\down}_{v,\ell})$ with $|S|\leq n' $ and $\deg_{T}(S) > C_1W$, there is also a set $I\subseteq [2N]$ of size $n'$ such that $\sum_i \bm d_i > C_1 W$.

Now we can use a union bound to upper bound the probability of this event. For any $S = \{i_1,\dots, i_{n'}\}\subseteq [2N]$, $\bm d_{i_1},\dots, \bm d_{i_{n'}}$ are i.i.d. $D + \Poisson(d)$,  hence their sum is distributed as $D n' + \Poisson(dn')$. Using $\Pr(\Poisson(\lambda)\geq x)\leq (\tfrac{\emm\lambda }{x})^x$, we get
\[\Pr\left(\sum_{i\in S} \bm d_i \geq C_1W\right)\leq \left(\frac{\emm dn'}{C_1 W - Dn'}\right)^{C_1 W-Dn'} = \exp\{-C_1 \Theta(\log n) \},\]
where we used the fact that $-\log \left(\frac{\emm dn'}{C_1 W - Dn'}\right) = \Theta(\log\log n)$. Pick $C_1$ large enough so that for all $n$ large enough the expression is at most $n^{-(\kappa + 2)}d^{-h}$.

Then, by noting that there are at most $\left(\frac{\emm 2N}{n'}\right)^{n'} = \exp\{\poly(\log\log n)\} = n^{o(1)}$ such subsets $S$, we conclude by a union bound that, for every sufficiently large~$n$,
\[
\Pr\big(\cE(T_{v,\ell}^{\down})\, \big| \, |T^{\down}_{v,\ell}|\leq N\big) \geq 1 - \tfrac 12 n^{-(\kappa+1)} d^{-h}.
\]
Finally, use the probability bound that we've already obtained from Lemma~\ref{lem:Tvl-polylog} to remove the conditioning to conclude that
\begin{align*}
\Pr\big(\cE(T^{\down}_{v,\ell})\big) \geq 
\Pr\big(\cE(T^{\down}_{v,\ell})\, \big| \,  |T^{\down}_{v,\ell}|\leq N\big) \Pr\big(|T^{\down}_{v,\ell}|\leq N\big) \geq 1 - n^{-(\kappa + 1)} d^{-h},
\end{align*}
for all sufficiently large~$n$. Combining everything together, we obtain that
\begin{equation} \label{eq:dgs1}
    \sup_{v \in V(T^h)} \Pr\big(\neg \cE_v\big) \leq n^{-(\kappa + 1)} d^{-h}.
\end{equation}

For all $v \in V(T^h)$, let ${\bf 1}_{\neg\cE_v}$ be the indicator r.v. for the event $\neg \cE_v$, and let $B \coloneqq \sum_{v \in V(T^h)} {\bf 1}_{\neg\cE_v}$. For all $0 \leq t \leq h$, let $Z_t$ be the number of vertices at depth $t$ in $T^h$. For a $\Poisson(d)$ Galton--Watson tree, $\E(Z_t) = d^t$. Therefore, in conjunction with (\ref{eq:dgs1}), for all $n$ sufficiently large enough,
\[\E(B) \leq \left(\sup_{v \in V(T^h)} \Pr\big(\neg \cE_v\big)\right) \left(\sum_{t = 0}^h \E(Z_t)\right) \leq \left(\frac{1}{n^{(\kappa + 1)} d^h}\right) \left(\frac{d^{h + 1} - 1}{d - 1}\right) \leq n^{-\kappa}\]

Then $\Pr\Big(\bigcap_{v\in V(T^h)}\cE_v\Big) = 1 - \Pr(B \geq 1) \geq 1 - \E(B) \geq 1 - n^{-\kappa}$, as required.
\end{proof}

\subsection{Sparse Random Graphs: Proof of Lemma~\ref{lem:Poissonneig}}\label{sec:gndnstruc}
\begin{lemPoissonneig}
\statelemPoissonneig
\end{lemPoissonneig}
\begin{proof}
Let $p=d/n$ and let $\lambda=-\log(1-p)$, and consider $G\sim G(n,d/n)$ and a vertex $v$ in it.

We expose the neighbourhood of $v$ by breadth-first search.  Let $s_1,s_2,\hdots$ be the vertices of $G$ that are explored. For $i=1,2,\hdots$, let $U_i$ be the set of vertices of $[n]$ which have not yet been discovered at the time when the vertex $s_i$. 
For each $u\in U_i$, we sample $P_{i,u}\sim \mathrm{Poi}(\lambda)$ and place the edge $\{s_i,u\}$ in $G$ if, and only if, $P_{i,u}\ge 1$. Since $\Pr(P_{i,u}\ge 1)=1-\emm^{-\lambda}=p=d/n$, 
this gives exactly the usual  BFS exploration on $G(n,d/n)$.

Let  $M_i \coloneqq \sum_{u\in U_i}P_{i,u}$. Conditional on $U_i$, we have  $M_s\sim \mathrm{Poi}(|U_i|\lambda)$. 
Since
\[
        |U_i|\lambda\le n\lambda=d+\tfrac{d^2}{2n}+O(1/n^2)\leq \hat d
\]
for all large enough $n$, we sample independently $R_i\sim \mathrm{Poi}(\hat d-|U_i|\lambda)$ 
and define $N_i \coloneqq M_i+R_i$.
Then, conditional on $U_i$, we have  $N_i\sim\mathrm{Poi}(\hat d)$. 

We use $N_i$ as the number of children of $s_i$ in the comparison Poisson tree
$T^h$. The BFS children of $s_i$ are embedded as follows. For every $u\in U_i$
with $P_{i,u}\ge 1$, choose one of the $P_{i,u}$ Poisson children and identify
the corresponding child of $s_i$ in $T^h$ with the BFS child $u$. The remaining
$P_{i,u}-1$ points, together with the $R_i$ additional children, are declared
to be the ``extra children'' (these extra children are then followed by
independent PGW$(\hat d)$ subtrees, truncated at the remaining depth).
Doing this recursively for all BFS vertices up to depth $h-1$ gives a coupling
in which $T_{v,h}$ is a rooted subtree of $T^h$. Moreover, the distribution of
$T^h$ is precisely that of a Poisson tree with parameter $\hat d$ truncated at depth $h$,
since every vertex has an independent $\mathrm{Poi}(\hat d)$ number
of children.

It remains to show that the number of extra children (attached directly to
vertices of $T_{v,h}$) is $O(1)$ with the probability $1-n^{-6/5}$. Let $\mathcal{F}_i$ be the BFS discovery tree at the time we started to explore $s_i$. Let $N=|V(T_{v,h-1})|$ be the number of vertices of $G$ that were explored.

For $i=1,2,\hdots$,  let $A_i$ be the
number of extra children created when $s_i$ is explored, i.e., with $P_i \coloneqq \sum_{u\in U_{i}}(P_{i,u}-1)\mathbf{1}\{P_{i,u}\geq 2\}$ we have $A_i \coloneqq R_i+P_i$. In case for some $i$ there does not exist a corresponding vertex $s_i$ (i.e., the BFS discovery stopped at some  vertex $s_j$ with $j<i$), we set $P_{i,u}=0$ for all $u\in U_i$ (so that $R_i=P_i=0$). Let $R=\sum^N_{i=1} R_{i}$, $P=\sum^N_{i=1} P_{i}$, and   $A \coloneqq R+P$.  It remains to bound $\Pr[A\geq 11]$, we  first bound $R$ and then $P$.  

With probability $1-\exp(-n^{\Omega(1)})$, we have $N\leq n^{2/5}$. Conditioned on this, for each $i$ we have $|U_i|\geq n-n^{2/5}$, so each $R_i$ is dominated above by a Poisson r.v. with mean $O(\tfrac{n^{2/5}}{n})$  and hence $R$ is dominated above by a Poisson  with mean $O(n^{2/5}\tfrac{n^{2/5}}{n})=O(n^{-1/5})$. It follows that  $\Pr[R\geq 7]\leq n^{-6/5}$ for large enough $n$. 

To bound $P$, consider a vertex $s_i$ in the discovery process and condition on $\mathcal{F}_i$. Under this conditioning,  for each $u\in U_{i}=V\backslash \mathcal{F}_{i}$, $P_{i,u}$ is distributed as
$\mathrm{Poi}(\lambda)$ and hence 
\[
        \Pr(P_{i,u}-1\ge2\mid \mathcal{F}_i)=O(1/n^3),
        \qquad
        \Pr(P_{i,u}-1\geq1\mid \mathcal{F}_i)=O(1/n^2).
\]
Note that these bounds extend to the case where there is no such vertex $s_i$.
Moreover, conditioned on $U_i$ the variables $\{P_{i,u}\}_{u\in U_i}$ are conditionally independent, so for any $i$ we have $\Pr(P_{i}\ge1\mid \mathcal{F}_i)=O( 1/n)$ and $\Pr(P_{i}\ge2\mid \mathcal{F}_i)=O( 1/n^2)$. Since this holds for any $\mathcal{F}_i$, we conclude that 
\begin{equation}\label{eq:Pi34t344t}
\Pr(P_{i}\ge1)=O( 1/n) \mbox{ and } \Pr(P_{i}\ge2)=O( 1/n^2).
\end{equation}
Similarly for $i<j$ we have that 
\begin{equation}\label{eq:Pi34t344tb}
\Pr( P_i\geq 1, P_j\geq 1)=O(1/n^2).
\end{equation}
Conditioned on the event $\mathcal{E}=\{N\leq n^{2/5}\}$, we  therefore 
have by a union bound that 
\[\Pr[P\geq 2\mid \mathcal{E}]\leq \sum_{1\leq i\leq n^{2/5}}\Pr\big[P_{i}\geq 2\mid \mathcal{E}\big]+\sum_{1\leq i<j\leq n^{2/5}}\Pr\big[P_{i}\geq 1,P_{j}\geq 1\mid \mathcal{E}\big].\]
Since $\Pr(\mathcal{E})=1-\emm^{-n^{\Omega(1)}}$, from \eqref{eq:Pi34t344t} and \eqref{eq:Pi34t344tb}  we have that $\Pr\big[P_{i}\geq 2\mid \mathcal{E}\big]=O(1/n^2)$ and $\Pr\big[P_{i}\geq 1,P_{j}\geq 1\mid \mathcal{E}\big]=O(1/n^2)$ as well, 
so $\Pr[P\geq 2]=O(\tfrac{n^{4/5}}{n^2})$. We therefore obtain that $\Pr[P\geq 2]=O(n^{-6/5})$. 

Combining the bounds for $R$ and $P$, we obtain that $\Pr[A\geq 11]=n^{-7/6}$ as needed.
\end{proof}

\section{Remaining Proofs}
\subsection{Proof for marginal bound}\label{sec:marginal}

In this section, we prove Lemma~\ref{lem:marginal}. Let $T$ be a tree rooted at $r$ and consider  the monochromatic boundary condition on $T$ at some level $h$. For $v\in T^h$, let $\Omega^{\pl}_{v}$ be the set of all colourings on $T_v$ when $v$ gets the colour $\pl$ and  $\Omega^{\mi}_{v}$ a set of all colourings  when $v$ gets some fixed colour $j\neq \pl$. Let $Z^{\pl}_v$ and $Z^{\pl}_v$ be the corresponding partition functions. Note $Z^{\pl}_v$ is the same regardless of the choice of $j\neq \pl$.

We will use the following recursive relation.
\begin{lemma}\label{lem:recurrence}
    For any non-leaf vertex \(v\in V(T)\), we have that
    \[
    \frac{Z^{\mi}_v}{Z^{\pl}_v} = \prod_{u:\, u\prec v} f\Big(\frac{Z^{\mi}_u}{Z^{\pl}_u}\Big) \mbox{ where } f(z) \coloneqq \frac{(q-2+\emm^\beta)z + 1}{(q-1)z+\emm^\beta},
    \]
    and \(u\prec v\) denotes that \(u\) is a child of \(v\).
\end{lemma}
\begin{proof}
    Note that once we fix the spin of \(v\), the distributions on its children's subtrees are independent, and the final partition function is a product of their partition functions.

    Thus, if $\sigma_v=j\neq i$,  there are $q-2$ choices for a non-$\{i,j\}$ colour for a child $u\prec v$, each of which contributing $Z^{\mi}_u$ on the conditional partition function on $T_u$ (since we do not add any additional monochromatic edge). 
    If $\sigma_u=j$, we instead obtain $\emm^\beta Z^{\mi}_u$. Finally, If $\sigma_u=j$, we get  $Z^{\pl}_u$. Consequently $Z^{\mi}_v = \prod_{u\prec v} \big[(q-2+\emm^\beta)Z^{\mi}_u + Z^{\pl}_u\big]$.     Analogously we have that \(Z^{\pl}_v = \prod_{u\prec v} \big[(q-1)Z^{\mi}_u + \emm^\beta Z^{\pl}_u\big]\). Combining, we obtain the equality in the lemma.
\end{proof}

Using this recurrence, we can show that for all vertices the bias (if any) is towards the boundary colour, and then, for good vertices, we establish \(\hat p \ll 1/q\). We start with the following observation.
\begin{observation}
    For $q\geq 2$ and $\beta>0$, the function $f(z) = \frac{(q-2+\emm^\beta)z + 1}{(q-1)z+\emm^\beta}$  is increasing.
\end{observation}
We thus obtain the following.

\begin{lemma}\label{lem:bias-towards-boundary-for-all}
    For all \(v\in V(T)\), \(Z^{\mi}_v\leq Z^{\pl}_v\).
\end{lemma}
\begin{proof}
    By induction on the distance from the leaves. For leaves $v$ of $T$ that are part of the boundary condition we have $Z^\pl_v = 1$ and $Z^\mi_v = 0$; for all other leaves $Z^\pl_v =Z^\mi_v$, so in both cases the lemma is true.  Assume the lemma holds for all \(u\prec v\) and consider the  expression for $\frac{Z^{\mi}_v}{Z^{\pl}_v}$  appearing in Lemma~\ref{lem:recurrence}. Then using that \(Z^{\mi}_u\leq Z^{\pl}_u\) for each \(u\prec v\), we have     \(f\Big(\frac{Z^{\mi}_u}{Z^{\pl}_u}\Big)\leq f(1)= 1\), so we conclude that \(\frac{Z^{\mi}_v}{Z^{\pl}_v}\leq 1\) as well.
\end{proof}

For large $q$, the next lemma shows a better upper bound of $\frac{Z^{\mi}_v}{Z^{\pl}_v}$ for all good vertices, and hence upper bounds $\hat p$.

\begin{lemmarginal}
\statelemmarginal
\end{lemmarginal}
\begin{proof}
Take  $d$ large enough so that $D(1-\delta)\geq \xi$ so that $\frac{\xi-1}{D(1-\delta)-1}\leq \frac{\xi}{D(1-\delta)}$. Assume that $q^{1-\xi} \leq \frac{\delta \xi}{2(1-\delta)D}\log q$ and note that the lower bound on $\beta$ implies that $q^{   1-\xi}\leq \beta \delta/2$ and $\emm^{\beta(1-D(1-\delta))}\leq q^{1-\xi} \leq \beta \delta/2$.

We first show by induction that for all good vertices $v$ it holds that $\frac{Z^{\mi}_v}{Z^{\pl}_v}\leq \emm^{-\beta D(1-\delta)}$. The bound is trivially true for leaves. Then, suppose \(v\) is a good vertex that is not a leaf and assume the bound is true for all the good children of \(v\). By Lemmas~\ref{lem:recurrence} and~\ref{lem:bias-towards-boundary-for-all},  using that $f(z) = \frac{(q-2+\emm^\beta)z + 1}{(q-1)z+\emm^\beta}$ satisfies $f(z)\leq f(1)=1$ for $z\leq 1$, we obtain
\begin{align*}
      \frac{Z^{\mi}_v}{Z^{\pl}_v} =\prod_{u\prec v} f\Big(\frac{Z^{\mi}_u}{Z^{\pl}_u}\Big)=&  \prod_{\substack{u\prec v \\ u\text{ good}}}f\Big(\frac{Z^{\mi}_u}{Z^{\pl}_u}\Big)\prod_{\substack{u\prec v \\ u\text{ bad}}}f\Big(\frac{Z^{\mi}_u}{Z^{\pl}_u}\Big) 
      \leq \prod_{\substack{u\prec v \\ u\text{ good}}} f(\emm^{-\beta D(1-\delta)}).
\end{align*}
Since $v$ is a good vertex, it has at least \(D \) good vertices amongst its children, so to finish the induction it suffices to show that  \(f(\emm^{-\beta D(1-\delta)})\leq \emm^{-\beta (1-\delta)}\). We have
\begin{equation}\label{eq:rf4rw1232}
f(\emm^{-\beta D(1-\delta)}) = \frac{(q-2+\emm^\beta)\emm^{-\beta D(1-\delta)} + 1}{(q-1)\emm^{-\beta D(1-\delta)} + \emm^\beta} \leq \frac{(q+\emm^\beta)\emm^{-\beta D(1-\delta)}+1}{\emm^\beta}.
\end{equation}
Using the assumption \(\emm^\beta\geq q^{\frac{\xi}{(1-\delta)D}}\) and the choice of $d,q$ at the start of the proof we have
\begin{equation}\label{eq:5gbb55}
   q\emm^{-\beta D(1-\delta)}\leq q^{1-\xi} \leq \delta\beta /2 \ \mbox{ and }\ 
    \emm^\beta\emm^{-\beta D(1-\delta)}=\emm^{\beta(1-D(1-\delta))}\leq  q^{1-\xi}\leq  \delta \beta/2.
\end{equation}
Taking the sum, we have $(q+\emm^\beta)\emm^{-\beta D(1-\delta)}\leq \delta\beta$, so we bound the r.h.s. in \eqref{eq:rf4rw1232}  by $\tfrac{\delta \beta+1}{\emm^{\beta}}\leq \tfrac{\emm^{\beta\delta}}{\emm^{\beta}}\leq \emm^{-\beta(1-\delta)}$, as needed.

To conclude the desired bound on $\hat p$, note first that the partition function for the Potts on $T_v$ is at least $Z^{\pl}_v$. Moreover,  when \(\sigma_{\mathrm{parent}(v)}=j\neq \pl\) we have \(q-2\) colours for $v$ other than $\pl$ that do not match $j$, and a single way to match the colours $j$ of the parent, contributing  a factor of \(\emm^\beta\) (because of the monochromatic edge). We therefore get:
\[
\hat p \leq \sup_{v\in V_\mathrm{good}} \frac{(q-2+\emm^\beta)Z^{\mi}_v}{Z^{\mi}_v}\leq \emm^{-\beta(1-\delta)D} (\emm^\beta + q)\leq 2q^{1-\xi},
\]
where in the last inequality we used the bounds in \eqref{eq:5gbb55}. 
\end{proof}

\subsection{Proof for local mixing via centroid decomposition}
\label{sec:local mixing}

In this section, we prove Lemma~\ref{lem:localspectral} using a standard block dynamics root-to-leaf reduction. To compensate for the fact that the sum of degrees might be large, we use a centroid decomposition, allowing for every induction path to be short (of order $O(\log\log n)$). A similar approach was used by Chen \cite{doi:10.1137/1.9781611977912.179} by utilising analogous structural decompositions based on tree-width and planarity. For completeness we give a more direct proof in the present setting. For a tree $T$, we also write $|T|$ as shorthand for $|V(T)|$.

 In order to prove Lemma \ref{lem:localspectral}, it will be convenient to work with the discrete-time and continuous-time versions of the single-site dynamics and the block dynamics. 
In the discrete-time dynamics, a vertex (respectively, block) is chosen u.a.r.{} at each step, and the corresponding heat-bath update is performed.
In the continuous-time dynamics, each vertex (respectively, block) is equipped with an independent rate-one Poisson clock and whenever its clock rings, the corresponding heat-bath update is performed. 
We use the following notation.
\begin{definition}
Given a graph $G = (V, E)$ and a set of vertices $B \subseteq V$, we define the \textit{vertex boundary} $\partial_{\mathrm{vert}} B$ to be the 
set containing all vertices in  $V \setminus B$ that are adjacent to some vertex in $B$.
Let $\trelct^{\mathrm{block}}$ 
be the relaxation time of the continuous-time block dynamics and let
$\trelct$ be the relaxation  time of the continuous-time single-site dynamics. For each block $B \in \mathcal{B}$ let $\trelct(B)$ be the relaxation time of the continuous-time single site dynamics on $B$, given any arbitrary condition on $\partial_{\mathrm{vert}} B$ and let $\trel(B)$ be the relaxation time of 
the discrete-time single site dynamics on $B$, given an arbitrary boundary condition on $\partial_{\mathrm{vert}} B$.
\end{definition}
We will use the following comparison result from \cite{Martinelli1999}, stated in the form appearing in~\cite{doi:10.1137/1.9781611975031.115}.

\begin{lemma}[{\cite[Proposition 3.4]{Martinelli1999}}, {\cite[Proposition 57] {doi:10.1137/1.9781611975031.115}}]\label{prop:block_rec_comp}
Let $G = (V, E)$ be a graph. Let $(X_t)_{t \in \mathbb{N}}$ be the continuous-time block dynamics, with set of blocks $\mathcal{B}$, where each vertex $v \in V$ belongs to $Q_v$ different blocks. Let $(Y_t)_{t \in \mathbb{N}}$ be the continuous-time single-site dynamics on $G$.  
Then   $\trelct \leq \trelct^{\mathrm{block}} \cdot \big(\max_{B \in \mathcal{B}} \trelct(B) \big) \cdot \big(\max_{v \in V} Q_v\big)$.
\end{lemma}

\begin{proof}[Proof of Lemma~\ref{lem:localspectral}]
	Let $\trelct(T')$ denote the relaxation times of the continuous-time single-site Glauber dynamics on $T'$. We will show that  $\trelct(T') \leq q^{O(\log N)} \emm^{O(\beta W)}$. Our claimed lower bound on the spectral gap of the discrete-time dynamics follows by noting that the relaxation time of the discrete-time dynamics is related to that of the continuous-time dynamics by $O(N) \cdot \trelct(T')$.

	We  will construct a block decomposition of $T'$.   A \textit{centroid} of a tree $H$ is a vertex $c$ such that, in the forest obtained from $H$ by deleting $c$, every component has size at most $|H| / 2$. It is well known that every tree has either one or two centroids (see e.g. \cite{harary1969graph}). If $T'$ is a single vertex, then $\trelct(T') = 1$ and we are done. Otherwise, fix an arbitrary centroid $c_0 $ in $T'$, and let $T^{(1)}, \dots, T^{(d_0)}$ be the components of $T'$ with $c_0$ deleted, where $d_0 \leq \deg_T(c_0)$. To ease the notation, let $T^{(0)}$ be the tree consisting of the single vertex~$c_0$.
		
	Consider the discrete-time block dynamics with blocks $\mathcal{B} = \left\{\{c_0\},V(T^{(1)}), \dots, V(T^{(d_0)})\right\}$.  To see that the mixing of this dynamics is at most $q^{O(1)} \emm^{O(\beta d_0)}$, consider a coupling of two copies of this dynamics, each making the same choices about which  blocks to update. At any step, the probability that block~$\{c_0\}$ is chosen next is $1/(d_0+1)$. If it is, then the probability that the spin at~$c_0$ is coupled is at least $q^{-O(1)} \emm^{-O(\beta d_0)}$. If this happens, then the probability all other blocks are chosen before $\{c_0\}$ is chosen again is $1/(d_0+1)$ and, if so, by the coupon collector argument, all of the other blocks are likely to be selected (and coupled) within $O(d_0 \log d_0)$ steps. 
		
	The relaxation time of the continuous-time variant of this block dynamics is at most the mixing time of the discrete-time variant, which is $q^{O(1)} \emm^{O(\beta d_0)}$. By Lemma~\ref{prop:block_rec_comp}, it follows that $\trelct(T') \leq q^{O(1)} e^{O(\beta d_0)} \max_{0 \leq i \leq d_0} \trelct(T^{(i)})$. 
		
	Similarly, for each $i \in [d_0]$, fix a centroid $c_{0, i}$ in $T^{(i)}$, and let $T^{(i, 1)}, \dots, T^{(i, d_{0, i})}$ be the components of $T^{(i)}$ with $c_{0, i}$ deleted, where $d_{0, i} \leq \deg_{T}(c_{0, i})$.  By Lemma~\ref{prop:block_rec_comp}, $\trelct(T^{(i)} \leq q^{O(1)} \emm^{O(\beta d_{0,i})} \max_{0 \leq j \leq d_{0, i}} \trelct(T^{(i,j)})$. Combining, we obtain the following: \[\trelct(T') \leq \max_{0 \leq i \leq d_0} \max_{0 \leq j \leq d_{0, i}} q^{O(1)} \emm^{O(\beta (d_0 + d_{0,i}))} \trelct(T^{(i,j)}).\]
		
	We continue this procedure recursively, where a particular path $\p = (i_1, i_2, \dots)$ of the process terminates when the component $T^{(\p)}$ is a single vertex, in which case $\trelct(T^{(\p)})  = 1$. Let $\mathcal{P}$ be the set of all ``root-to-leaf'' paths in the process and for all $\p \in \mathcal{P}$, let $\mathcal{C}(\p)$ be the set of centroids chosen along the path $\p$. We therefore obtain that
	\[\trelct(T') \leq \max_{\p \in \mathcal{P}} q^{O(|\mathcal{C}(\mathbf{p})|)}\emm^{O(\beta \deg_T (\mathcal{C}(\mathbf{p})))}.\]
		
	Since, upon deletion of a centroid from a tree, the maximum component size is at most half the size of the tree,  the process continues for at most $\lceil \log_2 N \rceil$ iterations, i.e. every path $\p \in \mathcal{P}$ corresponds to a collection $\mathcal{C}(\p)$ of at most $\lceil \log_2 N \rceil$ vertices in $T'$. Since for every subset of vertices with $|S| \leq \lceil \log_2 N \rceil$ it holds that $\deg_T (S) \leq W$, it follows that $\trelct(T') \leq q^{O(\log N)} \emm^{O(\beta W)}$, as required.  
\end{proof}

\subsection{Proof of Lemma~\ref{lem:log-sob-RC-dynamics}}\label{sec:Paulina}
In this section, we prove Lemma~\ref{lem:log-sob-RC-dynamics}. We will need the following lemma that captures the tree-like properties of $G(n,d/n)$.
\begin{lemma}[{\cite[Lemma 2.2]{BlancaGheissari2023}}]\label{lem:treelike}
There is  $L>0$ such that for all real $d>1$ and $n$ sufficiently large the following holds. Let $h=\lfloor\tfrac{1}{5}\log_d n\rfloor$.

With probability $1-1/n^{\Omega(1)}$ over the choice of $G(n,d/n)$,
for any vertex $v$, the $h$-step neighbourhood around $v$ can be made into a tree by removing at most $L$ edges.
\end{lemma}

\begin{proof}[Proof of Lemma~\ref{lem:log-sob-RC-dynamics}]
    Let $\hat d = d + \tfrac{d^2}{n}$ and $L$ as in Lemma~\ref{lem:treelike}. Fix $\theta > 1$ and let $d$ be large enough for Lemma~\ref{lem:conversion} to hold with $K = \tfrac{1}{5}$, $\kappa = 2$. Assume $n\geq d^2$ and $q \geq (d+1)^\theta \geq \hat d^\theta$ be integers and $\beta > (1+\tfrac{1.01}{\theta})\tfrac{\log q}{d}$.

    By Lemma~\ref{lem:conversion}, with probability $1 - 1/n^2$, the log-Sobolev constant on the Poisson tree $T^h$ with parameter $\hat d$ truncated to the first $\lfloor\tfrac15\log_d n\rfloor$ levels is $\geq \frac{1}{|E(T^h)|n^{o(1)}}$. Note that while the lemma (and the other results in this work) are stated for $d$ constant, one can easily verify that our results hold for $\leq \hat d(n)$, which is bounded by a constant (in our case $d\leq \hat d(n)\leq d + 1$).

    Furthermore by Lemmas~\ref{lem:Poissonneig} and~\ref{lem:treelike} and union bound, with probability $1-n^{-1/6}$ over $G\sim G(n,d/n)$, there are $n$ (not necessarily independent) $\Poisson(\hat d)$ trees truncated at level $h$, $T_{1}^h,\dots,T_n^h$ such that for all $v\in [n]$, $B_{v,h}$ is $T_v^h$ up to modifying at most $10 + L$ additional edges.

    Thus in particular, with probability $1-n^{-\Omega(1)}$, for every $v$, there is a set of at most $10 + L$ edges $H_v$, such that
    $B_{v,h} \oplus H_v$ is exactly $T_v^h$, and the log-Sobolev constant for the wired RC dynamics on $T_{v}^h$ is $\geq 1/(|E(T_v^h)|n^{o(1)})$.

    We next outline how to obtain the same asymptotic bound for the log-Sobolev constant for the wired RC dynamics on $B_{v,h}$ using standard lifting techniques (see also~\cite{Gheissari2022, BlancaGheissari2023}).

    We compare with yet another RC dynamics on a disjoint union of the $T_{v}^h$ with the wired boundary condition and disjoint free edges $H_v$ (i.e. we consider each edge of $H_v$ to be in a separate component). The Gibbs distribution in this setting is a product of the distribution of $\pi^{\pl}_{T^h;q,\beta}$ and the distributions of the individual free edges. It is well-known that entropy tensorises over product spaces, thus also the log-Sobolev constant of this new RC dynamics is at least $\min\{\Omega(1), 1/(|E(T_i^h)|n^{o(1)})\} =  1/(|E(T_i^h)|n^{o(1)})$. 

    Then we note that flipping a state of an edge changes the weight of the configuration by at most a factor of $q\emm^\beta$. Since $|H_v|\leq 10 + L$, we have that the marginals of any $F\subseteq E(B_{v,r})$ in the Gibbs distribution on the wired ball $B_{v,r}$ and that on the new product graph are within $(q\emm^{\beta})^{O(L)}$ multiplicative factor apart, and so are the transition probabilities of the corresponding RC dynamics.
    Consequently, the Dirichlet forms of both chains are a constant factor apart. Using the comparison result from~\cite[Theorem 4.1.1.]{Saloff-Coste1997LecturesChains}, we obtain that the Log-Sobolev constant of the wired ball $B_{r,v}$ is at least $\Omega\left(\tfrac{1}{|E(T_i^h)| n^{o(1)}}\right) = \frac{1}{|E(B_{r,v})| n^{o(1)}}$ (where the additional constant depends only on $q$ and $\beta$), which concludes the proof.
\end{proof}

\end{document}